\documentclass[a4paper,10pt]{article}

\usepackage[T1]{fontenc}

\usepackage[english,francais]{babel}

\usepackage{amsmath,amsthm,amssymb}
\usepackage{graphicx}
\usepackage[all]{xy}
\usepackage{psfrag}

\usepackage{color}
\usepackage{epsfig}
\usepackage{latexsym}

\newtheorem{thm}{Theorem}[section]
\newtheorem{cor}[thm]{Corollary}
\newtheorem{lem}[thm]{Lemma}

\newtheorem{prop}[thm]{Proposition}
\newtheorem{notn}[thm]{Notation}
\newtheorem{defn}[thm]{Definition}

\theoremstyle{definition}
\newtheorem{rem}[thm]{Remark}
\newtheorem{rems}[thm]{Remarks}
\newtheorem*{acknowledgement}{Acknowledgements} 

\numberwithin{figure}{section}

\newcommand{\bs}{\symbol{92}}

\newcommand{\cerc}{{\mathbb{S}^1}}
\newcommand{\cc}{\mathbb{C}}
\newcommand{\rr}{\mathbb{R}}
\newcommand{\zz}{\mathbb{Z}}

\newcommand{\ctg}{T^{\ast}}

\newcommand{\hnu}{\widehat{\nu}}
\newcommand{\sph}{\mathbb{S}}

\newcommand{\cU}{\mathcal{U}}

\newcommand{\cM}{\mathcal{M}}
\newcommand{\cL}{\mathcal{L}}
\newcommand{\cV}{\mathcal{V}}
\newcommand{\cH}{\mathcal{H}}

\newcommand{\bv}{\overline{v}}
\newcommand{\uu}{\underline{u}}
\newcommand{\uv}{\underline{v}}
\newcommand{\uw}{\underline{w}}

\newcommand{\tc}{\widetilde{c}}
\newcommand{\td}{\widetilde{d}}
\newcommand{\ti}{\widetilde{\imath}}
\newcommand{\tell}{\widetilde{\ell}}
\newcommand{\tp}{\widetilde{p}}
\newcommand{\tz}{\widetilde{z}}
\newcommand{\tx}{\widetilde{x}}
\newcommand{\ty}{\widetilde{y}}

\newcommand{\bL}{\bar{L}}
\newcommand{\bA}{\bar{A}}
\newcommand{\tL}{\widetilde{L}}
\newcommand{\tcL}{\widetilde{\mathcal{L}}}
\newcommand{\tM}{\widetilde{M}}
\newcommand{\tcM}{\widetilde{\mathcal{M}}}
\newcommand{\tJ}{\widetilde{J}}
\newcommand{\tgamma}{\widetilde{\gamma}}
\newcommand{\tvarphi}{\widetilde{\varphi}}

\newcommand{\hv}{\widehat{v}}
\newcommand{\hJ}{\widehat{J}}
\newcommand{\hcM}{\widehat{\mathcal{M}}}
\newcommand{\hLambda}{\widehat{\Lambda}}

\DeclareMathOperator{\im}{Im}

\DeclareMathOperator{\modulo}{mod}

\DeclareMathOperator{\Cal}{Cal}

\title{Obstructions to the existence of monotone Lagrangian embeddings
    into cotangent bundles of manifolds fibered over the circle} 
\author{Agn\`es GADBLED
}
\date{\today}

\begin{document}
\selectlanguage{english}
\maketitle

\begin{abstract}
 We extend the constructions and results of Damian \cite{Mihai} to get
 topological obstructions to the existence of closed monotone
 Lagrangian embeddings into the cotangent bundle of a space which is
 the the total space of a fibration over the circle.
\end{abstract}

\section{Introduction}
\label{sec:intro}

Let $M$ be a closed manifold and $\pi: \ctg M \rightarrow M$ its
cotangent bundle. Denote by $\lambda_M$ the Liouville one-form of
$M$ and  $\omega_M= d \lambda_M$ the canonical symplectic form on
$\ctg M$.

We are interested in compact Lagrangian submanifolds in the cotangent
bundle $\ctg M$. Only a few types of examples are known:
\begin{enumerate}
\item the zero section or more generally the graph $L_{f}$ of a
  function $f~:~M~\rightarrow~\rr$;
\item \label{item:2}any Hamiltonian image of $L_{f}$
  (i.e. $L=\varphi_{1}(L_{f})$ where $(\varphi_{t})$ is a Hamiltonian  
  isotopy);
\item any image of $L_{f}$ by a symplectic isotopy (as in \ref{item:2} but
  with $(\varphi_{t})$ a symplectic isotopy);
\item the ``local'' Lagrangian submanifolds: any Lagrangian
  submanifold of $\cc^n$ can be embedded in a Darboux chart 
  $U \stackrel{\sim}{\rightarrow} \cc^{n}$ of $\ctg M$. 
\end{enumerate}

Note that the two first types of examples have the additional property
of being exact (that is, the restriction of the Liouville one-form on
the Lagrangian submanifold is exact). It is conjectured (see
\cite{MR890489}) that the examples (1) and (2) are the only possible
examples of exact Lagrangian submanifolds in $\ctg M$ and much work
has been done to prove this conjecture. It has been proved for 
$L = M = S^{2}$ (Hind \cite{MR2060197}) but in general, we know only
topological restrictions on the exact Lagrangian embeddings $i: L
\rightarrow \ctg M$ of a closed manifold~$L$. In the following, 
$f : L \rightarrow M$ denotes the composition $\pi \circ i$.

\begin{enumerate}
\item (Audin \cite{MR966952}) If $L$ and $M$ are orientable, then 
  $\chi(L) = \deg(f)^2 \chi(M)$ and the same is true 
  modulo $4$ if $L$ and $M$ are not orientable.
\item (Lalonde and Sikorav \cite{MR1090163}) The index
  $[\pi_{1}(M):f_\star (\pi_{1}(L))]$ is finite.
\item (Viterbo \cite{MR1726235}) If $M$ is simply connected, then
  $L$ cannot be an Eilenberg-MacLane space.
\item (Fukaya, Seidel and Smith \cite{MR2385665}, Nadler \cite{Nadler},
  see also Buhovski \cite{Lev}) If $M$ is simply connected and $L$ is
  spin with zero Maslov class, then the projection $f$ has degree 
  $\pm 1$, and induces an isomorphism $H^{\star}(L,K) \simeq
  H^{\star}(M,K)$ for any field $K$ of characteristic not equal to~$2$.
\item (Ritter \cite{Ritter}) If $M$ is simply connected, then 
  $\pi^\star: H^{2}(M) \rightarrow H^{2}(L)$ is injective 
  and the image of $f_\star: \pi_2(L) \rightarrow \pi_2(M)$ 
  has finite index.
\item (Damian \cite{Mihai}) If $M$ has dimension $n \geq 3$ and is the 
  total space of a fibration over $\cerc$, we have:\\
  a) For any finite presentation 
  $\langle g_{1}, g_{2}, \ldots, g_{p}\, |\, r_{1}, r_{2}, 
  \ldots, r_{q} \rangle$ of the fundamental group $\pi_{1}(L)$, 
  $p-q\leq~1$. \\
  b) The fundamental group $\pi_{1}(L)$ is not isomorphic to the free 
  product $G_{1}\ast G_{2}$ of two non-trivial groups.
\end{enumerate}

The idea of the proof of Damian is the following. On the one hand, if
$M$ is the total space of a fibration $p : M \rightarrow \cerc$ over
the circle, then one can use the pull-back $\alpha = p^\star d \theta$
of the one-form $d \theta$ on $\cerc$ to displace an exact Lagrangian
submanifold $L$ of $\ctg M$ from itself by a symplectic isotopy
(defined by $\varphi_t (q,p) = (q , p + t \alpha)$).

On the other hand, given a symplectic isotopy $(\varphi_t)$, Damian
has constructed a Floer-type complex $C(L,\varphi_t)$ spanned by the
intersection points of $L$ and $\varphi_1(L)$, with coefficients in a
Novikov ring associated to $\pi_1(L)$ and endowed with a differential
which is an analogue in the Lagrangian Floer theory of the
Morse-Novikov differential. The homology of this complex only depends
on the flux $[ \varphi_1^\star \lambda_M - \lambda_M ] = u$ of
$(\varphi_t)$. It is called the Floer-Novikov homology of $L$ and
denoted $FH(L,u)$. 

Damian has proved that this homology $FH(L,u)$ 
is isomorphic to the Novikov homology $H_\star(L,f^\star u)$ 
of $L$. The Novikov homology of $L$ must then be trivial when 
$\varphi_t$ is the isotopy induced by $\alpha$, that is when 
$u = [\alpha]$, and this gives the obstructions on the fundamental 
group of $L$.\\

Here we are interested in the more general case of monotone Lagrangian
submanifolds of $\ctg M$. In the usual sense, a Lagrangian submanifold
$L$ is monotone (on the disks) if there exists a non-negative constant
$K_L$ such that: 
 \begin{equation}
    \label{eq:MonD}
    \text{for all } u \in \pi_2(\ctg M,L), \int_u \omega_M = K_L \;
    \mu_L (u)
  \end{equation}
where $\mu_L$ denotes the Maslov class of $L$ in $\ctg M$.

For instance, any local Lagrangian submanifold which is monotone in  
$\cc^n$ is also monotone in $\ctg M$ (see Remark \ref{reme:sphmon} (iii)).

We would like to know if there are also ``global'' monotone 
Lagrangian submanifolds in $\ctg M$. It is possible to get 
topological obstructions on the monotone Lagrangian embeddings by 
extending the construction of the 
Floer-Novikov type complex of Damian to the monotone case. In order 
to carry out this construction, we need a stronger monotonicity 
assumption:

\begin{defn}
\label{defn:mon}
  A Lagrangian submanifold of $\ctg M$ is said to be monotone on the 
  loops if there exists a non-negative constant $k_L$ such that:
  \begin{equation}
    \label{eq:MonL}
    \text{for all } \gamma \in \pi_1(L), \int_{\gamma} \lambda_M = 
    k_L \; \mu_L (\gamma)
  \end{equation}
where $\mu_L$ denotes the Maslov class of $L$ in $\ctg M$.
\end{defn}

\begin{rems}
\label{reme:sphmon}
\begin{description}
\item[(i)] We recover the exact case when $k_L = 0$.
\item[(ii)] If $L$ is a Lagrangian submanifold of $\ctg M$ which is
  monotone on the loops, then it is monotone in the usual sense
  (i.e. on the disks) with the same constant $k_L = K_L$.
\item[(iii)] Note that the converse of \textbf{(iii)} is not
  necessarily true in general. It is true for instance if $M$ is
  simply connected. 
\item[(iv)]  This definition of monotone Lagrangian submanifold was
  already used by Polterovich \cite{MR1109663} for Lagrangian
  submanifolds of $\cc^n$, but in this case, it coincides with the
  usual definition.
\item[(v)] This assumption is necessary to bound from above the energy
  of solutions having the same Maslov index (see
  Section~\ref{sec:Floerdiff}).
\end{description}
\end{rems}

\begin{notn}
  If $L$ is a Lagrangian submanifold in $\ctg M$ which is monotone on
  the disks, we will call Maslov number of $L$ and denote $N_{L}$ the 
  non-negative generator of the subgroup 
  $\langle \mu_{L} , \pi_{2}(L) \rangle$ of $\zz$.
\end{notn}

As in the exact case, we will use a suitable version of Floer-Novikov
homology. The differences are the following: 
\begin{description}
\item[(i)] There is no action functional, so we will have to work with
  the action one-form (see Remark~\ref{rem:fonctionelle}).
\item[(ii)] The fact that $[\pi_1(M): f_\star (\pi_1(L))]$ is finite 
does not hold 
in the monotone case without further assumption on the 
Maslov class of $L$. For instance, for any local monotone Lagrangian 
submanifold in a Darboux chart, 
$f_\star : \pi_1(L) \rightarrow \pi_1(M)$ is 
trivial and the index is not finite unless $\pi_1(M)$ is finite.
\item[(iii)] In the monotone case, we have to take into account the 
bubbling of $J$-holomorphic disks. This makes the definition of
Floer-Novikov homology more intricate.
\item[(iv)] There are also differences on more technical points. 
For instance, in the proof of invariance, we cannot use 
an extension of a symplectic isotopy of $\ctg L$ to $\ctg M$ as in 
the exact case.
\item[(v)] In the monotone case, the Floer homology is not always 
isomorphic to the Novikov homology $H_\star(L,f^\star u)$. We will 
prove that it is the limit of a spectral sequence 
(see~Theorems~\ref{thm:suitespecintro} and~\ref{thm:suitespec}, this is
a ``Novikov version'' of the spectral sequence described by Biran
in~\cite{Birannew} for the usual Lagrangian Floer homology).
\end{description}

\begin{thm}
\label{thm:suitespecintro}
Let $u$ be an element of $H^1(M)$.
Assume that the Lagrangian submanifold $L$ is monotone on the loops, 
and of Maslov number $N_L \geq 2$.
There exists a spectral sequence $\{E_{r}^{p,q},d_{r}\}$ satisfying 
the following properties:
\begin{description}
\item[(1)] $E_{0}^{p,q} = C_{p+q-pN_L}(L,f^\star u) \otimes A^{pN_{L}}$ and 
$d_0=\partial_0 \otimes 1$, where $\partial_0$ is the Morse-Novikov 
differential;
\item[(2)] $E_{1}^{p,q} = H_{p+q-pN_L}(L,f^\star u) \otimes A^{pN_{L}}$ and  
$d_1=[\partial_1] \otimes \tau$ where 
$$[\partial_1]: H_{p+q-pN_L}(L,f^\star u) \rightarrow 
H_{p+1+q-(p+1)N_L}(L,f^\star u) \,;$$
\item[(3)] $\{E_{r}^{p,q},d_{r}\}$ collapses at the page $\kappa + 1$, 
  where $\kappa = [\frac{\dim(L)+1}{N_L}]$ and the spectral sequence 
  converges to $FH(L,u)$, i.e.  
  $$\bigoplus_{p+q=\ell} E_{\infty}^{p,q} \cong FH^{\ell(\modulo
    N_{L})}(L,u).$$
\end{description}
\end{thm}

Here $A$ is the $\Lambda_{f^\star u}$-module 
$A = \Lambda_{f^\star u}[T,T^{-1}]$ of Laurent 
polynomials with coefficients in $\Lambda_{f^\star u}$ 
(see section~\ref{sec:Nov}) 
and $\tau_i :A \rightarrow A$ 
is the multiplication by~$T^i$. The degree of $T$ is 
equal to $N_L$ and $A^{pN_L} =  \Lambda_{f^\star u} \, T^p$. \\

In particular, when $N_L \geq \dim(M) + 2$, the spectral sequence 
above collapses at the first page and the Floer-Novikov homology is 
equal to the Novikov homology $H(L,f^\star u)$.\\

Using the techniques of Damian, we prove:

\begin{thm}
\label{thm:princ}
Let $M$ be a closed manifold which is the total space of a fibration 
$p: M \rightarrow \cerc$ on the circle.
Let $L$ be a Lagrangian submanifold of $\ctg M$ which is monotone 
on the loops.
Assume either that $N_L \geq \dim(M)+1$ or ($N_L = \dim(M)$ and 
$[\pi_1(M):\pi_1(L)]$ is finite). We have:
\begin{description}
\item[(i)] If $\langle g_1,g_2, \dots , g_p | r_1,r_2, \dots , r_q \rangle$ 
is a finite presentation of $\pi_1(L)$, then $$p-q \leq 1.$$
\item[(ii)] The fundamental group $\pi_1(L)$ is not isomorphic to a  
free product $G_1 \star G_2$ of two non trivial groups.
\end{description}
\end{thm}

This theorem will be proved in section \ref{sec:FloerNov}.

\begin{acknowledgement} 
  I would like to thank Mich\`ele Audin and Mihai Damian for
  suggesting me to work on this extension of Floer homology, for
  valuable discussions around this subject and for their careful
  reading. I would also like to thank Jean-Claude Sikorav for his
  numerous remarks on the work which has given rise to this article. 
\end{acknowledgement}

\section{Novikov theory}
\label{sec:Nov}

Let us recall the definition of Novikov homology~(\cite{MR630459}, for 
a more detailed study, see \cite{TheseSiko}) and the results of
Damian~\cite{Mihai}.\\

Let $L$ be a closed manifold and $u \in H^{1}(L,\rr)$. 
Denote by $\Lambda$ the ring $\zz/2 [\pi_{1}(L)]$ and by
$\hLambda = \zz/2 [[\pi_{1}(L)]]$ the group of formal series.\\
Let $\Lambda_{u}$ be the completed ring of series 
$$\Lambda_u= \left \{\sum n_{i} g_{i} \in \hLambda \left | \; 
  g_{i} \in \pi_{1}(L), n_{i} \in \zz/2, u(g_i) \rightarrow + \infty \right. 
  \right \}$$
where $u(g_i) \rightarrow + \infty$ means here that for all $A > 0$, the 
set $$\{g_{i}\mid n_{i} \neq 0 \mbox{ and } u(g_{i}) <  A \}$$ 
is finite.

\begin{defn}
  Let $C(\tL)$ be the free $\Lambda$-complex spanned by fixed lifts 
  of the cells of a CW-decomposition of $L$ 
  to the universal cover $\tL$ of $L$  and 
  $$C(L,u) = \Lambda_{u} \otimes_{\Lambda} C(\tL).$$

  The homology of this complex  $C(L,u)$ is the Novikov homology
  $H(L,u)$.
\end{defn}

\begin{defn}[Morse-Novikov homology]
Let $\alpha$ be a closed generic one-form in the class of $u \in
H^{1}(L,\rr)$ and $\xi$ be the gradient of $\alpha$ 
with respect to some generic metric on $L$. 
For every critical point $c$ of $\alpha$, fix a lift $\tc $ of $c$ in
the universal cover $\tL$.

Let $C(\alpha,\xi)$ be the $\Lambda_{u}$-complex spanned by the zeros 
of $\alpha$ and whose differential is such that if $c$ and $d$ are
zeros of index difference equal to $1$ then their incidence number 
is the algebraic number of flow lines that joins $c$ to $d$ and lifts
to a path in $\tL$ from $g_{i} \tc$ to $\td$.
\end{defn}

By generic, we mean here that $\alpha$ has Morse-type singularities and 
$\xi$ satisfies the Morse-Smale condition. 

\begin{thm}[Novikov \cite{MR630459}, Latour~\cite{MR1320607}, 
see also Pazhitnov~\cite{MR1404410}]
For any generic pair $(\alpha,\xi)$, the homology of this complex is
isomorphic to $H(L,u)$.  
\end{thm}

\begin{rem}
\label{rem:revbar}
In these two descriptions of the Novikov homology, one could
replace $\tL$ by any integration covering $\bL$ of $L$ (i.e. such
that the pull-back of $u$ is zero).  We will use this in our
comparison between the Floer-Novikov homology and the
Novikov homology of $L$ in Section~\ref{sec:FloerNov}.\\
\end{rem}

Let us end this section by recalling results proved in \cite{Mihai}:

\begin{prop}[Damian {\cite{Mihai}}, Sikorav {\cite{TheseSiko}}]  
\label{prop:prelibre}
  Let $L$ be a closed manifold and $u~\in~H^{1}(L,\rr)$. 
  \begin{description}
  \item[a)] Let $\langle g_{1}, g_{2}, ..., g_{p} | r_{1}, r_{2}, ...,
    r_{q} \rangle$ be a presentation of the fundamental group
    $\pi_{1}(L)$ which satisfies $p-q \geq 2$. If $u \neq 0$, then 
    $H_{1}(L,u) \neq 0$.
  \item[b)] Suppose that $\pi_{1}(L) = G_{1} \ast G_{2}$ is a free
    product of two groups, none of them being trivial. If $u \neq 0$,
    then $H_{1}(L,u) \neq 0$.
  \end{description}
\end{prop}

\section{The Floer type complex}
\label{sec:cpxFloer}

Let $i: L \hookrightarrow \ctg M$ be a Lagrangian embedding of a 
closed manifold $L$. Assume that the image of $L$ in $\ctg M$ is 
monotone on the loops. 
Let $(\varphi_{t})$ be a symplectic isotopy of $\ctg M$ and 
denote by $L_{t}$ the image of $L$ by $(\varphi_{t})$.

\begin{rem}
  As they are the images of $L$ by symplectomorphisms, the Lagrangian
  submanifolds $L_{t}$  are not necessarily monotone on the
  loops but they are monotone on the disks. We can therefore apply to 
  them all the results concerning monotone (in the usual 
  sense) Lagrangian submanifolds.
\end{rem}

Let $u \in H^1(M;\rr)$ denote the flux (or Calabi invariant) of
$(\varphi_{t})$, that is the class:
$$\Cal(\varphi_t)=[\varphi_1^\star\lambda_M-\lambda_M] \in H^1(\ctg M,\rr)
\simeq H^1(M,\rr).$$

The results proved in Sections \ref{sec:respre} to \ref{sec:Floerdiff} 
are used in Section \ref{sec:Floerdiff} to define a
$\Lambda_u$-complex $C(\bL_0,\bL_1)$ spanned by the intersection
points of $L_0=\varphi_0(L)$ and $L_1=\varphi_{1}(L)$. We prove in
Section~\ref{sec:invHam} that this complex depends only on $L$ and $u
= \Cal(\varphi_t)$. We then explain in Section~\ref{sec:surfstaru}
that these results can be used to define a free complex
over the Novikov ring $\Lambda_{f^\star u}$ spanned by $L \cap
\varphi_1(L)$ whose homology only depend on $L$ and $u$.

\subsection{Preliminary results}
\label{sec:respre}

Let us first notice that it is possible to restrict our symplectic 
isotopies to isotopies of the following type:

\begin{lem}
\label{lem:3.1}
 There exists a symplectic isotopy $(\psi_t)$ on $\ctg M$ such that 
$\psi_{1|L} = \varphi_{1|L}$ which is spanned by $\alpha + d H_t$,
where $\alpha$ is a closed one-form in $u$ and 
$H: \ctg M \times [0,1] \rightarrow \rr$ has compact support.
\end{lem}

\begin{proof}
  As in the proof of \cite[Lemma 3.2]{Mihai} (which does not use the 
exactness assumption on $L$), consider a family of one-forms
$\alpha_t$ on $M$ in the class of $\varphi_t^\star \lambda_M -
\lambda_M$. Note that the composition of $(\varphi_{t})$ and of the
symplectic isotopy spanned by $-\alpha_t$ is a compactly
supported Hamiltonian isotopy $\chi_t$. The isotopy $(\psi_t)$ can then be defined
as the composition of $\chi_t$ with the symplectic isotopy spanned by
$\alpha_1$. 
\end{proof}

We will also require in the construction and applications that the
induced map $f_{\star} : \pi_{1}(L) \rightarrow \pi_{1}(M)$ is
surjective. It is enough to suppose that the index  
$[\pi_{1}(M): f_\star (\pi_{1}(L))]$ is finite:

\begin{rem}
\label{rem:indice}
If $f_{\star}: \pi_{1}(L) \rightarrow \pi_{1}(M)$ is not surjective,
let $M_1$ be the covering of $M$ induced by the subgroup $f_{\star}
(\pi_{1}(L))$ of $\pi_{1}(M)$. The manifold $L$ can be lifted as a
Lagrangian submanifold of $\ctg M_1$ which is monotone on the loops
(and which has the same Maslov class). 

If  $[\pi_{1}(M): f_\star (\pi_{1}(L))]$ is finite, then $M_1$ is 
closed and the conclusion of Theorem~\ref{thm:princ} in $\ctg M$ 
is then just a consequence of the same theorem for $\ctg M_1$ where
the surjectivity condition is satisfied.  
\end{rem}

Although the index $[\pi_{1}(M): f_\star (\pi_{1}(L))]$ is always
finite in the exact case (Lalonde and Sikorav, 
\cite[Theorem 1 a)]{MR1090163}), this assumption is not always
fulfilled in the monotone case. However, the index is necessarily
finite for monotone Lagrangian submanifolds if the Maslov number of
$L$ is large enough: 

\begin{prop}
\label{prop:indsurj}
  Let $L$ be a Lagrangian submanifold of $\ctg M$ which is monotone 
  (on the disks). If $N_L \geq \dim(M) + 1$, then the index of 
  $\pi_1(L)$ in $\pi_1(M)$ is finite.
\end{prop}

\begin{proof}
If $\pi_{1}(L) \rightarrow \pi_{1}(M)$ is not surjective, consider
again the covering $M_1$ of $M$ induced by the subgroup 
$f_{\star} (\pi_{1}(L))$ of $\pi_{1}(M)$ and the lift of $L$ into
$\ctg M_{1}$.

If the covering group of $M_{1} \rightarrow M$ is infinite, then
$M_{1}$ is open and the Lagrangian submanifold $L$ can be displaced 
from itself by a Hamiltonian isotopy (see \cite[Proposition 1]{MR1090163}) 
so that the usual Floer homology $HF(L,L)$ (with $\zz/2$ coefficients)
must be trivial. But if the Maslov number of $L$ is greater than 
$\dim (M) + 2$, we know by Oh's theorem (\cite[Theorem II (i)]{MR1389956}) 
that this homology is isomorphic to the usual cohomology 
$H^{\star}(L,\zz/2)$ of $L$ and this is in contradiction with the 
vanishing of $HF(L,L)$.

If $N_{L}= \dim(M)+1$, then by \cite[Theorem II (ii)]{MR1389956}, 
$$ HF(L,L)  \approx \bigoplus_{i=0}^{\dim(M)} H^{i}(L,\zz/2) \mbox{ or
  }\bigoplus_{i=1}^{\dim(M)-1} H^{i}(L,\zz/2)$$ and this also leads to
a contradiction unless $L$ is a $\zz/2$-homology sphere.

However, if $L$ is a $\zz/2$-homology sphere, then $H^1(L,\rr)=0$ and
in particular $L$ is exact so that we can directly apply the result of
Lalonde and Sikorav to see that the index is finite.
\end{proof}

\begin{rem}
  We cannot expect to remove the assumption $N_L \geq \dim(M)+1$
  in Proposition~\ref{prop:indsurj}. Indeed,   Polterovich proved in
  \cite{MR1109663} (see also Audin \cite{MR966952} for the
  construction) that for every two integer numbers $2 \leq k \leq n$,
  there exists a compact manifold $L_{n,k}$ which admits a monotone
  Lagrangian embedding in $\cc^{n}$ (and consequently ``local''
  monotone Lagrangian submanifolds in any cotangent bundle) with
  Maslov number equal to $k$ such that:
  \begin{description}
  \item[a)] $L_{n,n} = \sph^{n-1} \times \sph^{1} / \tau_{n-1} \times
    \tau_{1}$ where $\tau_{j}: \sph^{j} \rightarrow \sph^{j}$ is the
    antipodal involution;
  \item[b)] $L_{n,k} = L_{k,k} \times \sph^{n-k}$ where $k < n$.   
  \end{description}
\end{rem}

\begin{rem}
\label{rem:hypsurj}
  Thanks to Remark~\ref{rem:indice} and
  Proposition~\ref{prop:indsurj}, it is enough to prove
  Theorem~\ref{thm:princ} when the induced map 
  $$f_{\star}: \pi_{1}(L) \longrightarrow \pi_{1}(M)$$ 
  is surjective. From now on, we will always suppose that $f_{\star}$
  is surjective.
\end{rem}

Thanks to the surjectivity assumption on $f_\star$, we can consider a
connected covering of the Lagrangian submanifold $L$ in the cotangent
bundle of the universal cover of $M$:

\begin{lem}
\label{lem:rev}
  Let $\tM \rightarrow M$ be the universal covering of $M$ and 
$\tp: \ctg \tM \rightarrow \ctg M$ be the induced covering on the
cotangent bundles. Denote by $y \mapsto y^g$ the right action of
$\pi_1(M)$ on $\ctg \tM$. 

Let $\bL \rightarrow L$ be the pull-back of the covering  
 $\ctg \tM \rightarrow \ctg M$ by the embedding 
$i: L \rightarrow \ctg M$. Then:
\begin{description}
\item[(i)] $\bL$ is path-connected if and only if the map  
  $f_{\star} : \pi_1(L) \rightarrow \pi_1(M)$ is surjective.
\item[(ii)] The covering $\bL \rightarrow L$ corresponds to the
  covering of $L$ associated with the subgroup $K = \ker (f_{\star})$
  of $\pi_1(L)$. 
\item[(iii)] The map $\ti: \bL \rightarrow \ctg \tM$ is a monotone 
  Lagrangian embedding and for all $g \in \pi_1(M)$, $\tx \in \bL$, 
  $$\ti ( \tx^g) = \left({\ti(\tx)}\right)^g.$$
\end{description}
\end{lem}

The proof of this lemma uses only basic algebraic topology and can be
found in the appendix.\\ 

\begin{lem}
 Let $(\varphi_t)$ be the symplectic isotopy generated by 
$\alpha + d H_t$. Then this isotopy lifts to a Hamiltonian isotopy
$(\tvarphi_t)$ on $\ctg \tM$.

Moreover, if $L_t=\varphi_t(L)$ and $\bL_t=\tvarphi_t(\bL)$ , then
$$\bL \cap \bL_1 = \bigcup_{x \in L \cap L_1} \tp^{\: -1}(x).$$
\end{lem}

\begin{proof}
  As in the proof of \cite[Lemma 3.6]{Mihai}, which does not involve
  any assumption on the exactness/monotonicity of $L$, we can define
  the isotopy $(\tvarphi_t)$ as the isotopy spanned by the pullback
  of $\alpha + d H_t$ to $\ctg \tM$.
\end{proof}

\subsection{The action one-form}
\label{sec:fctaction}

In this section, $L$ is assumed to be a closed Lagrangian submanifold
of $\ctg M$. It is also assumed that $L$ is monotone on the loops and
that $f_{\star}: \pi_{1}(L) \rightarrow \pi_{1}(M)$ is surjective.

Let $(\varphi_t)$ be a symplectic isotopy as in Lemma~\ref{lem:3.1}. 
Denote by $u \in H^1(M;\rr)$ its flux.

If $L_t=\varphi_t(L)$, let $\Omega(L_0,L_1)$ be the space of paths 
from $L_0$ to $L_1$:
$$\Omega(L_0,L_1) = \{z \in C^\infty([0,1];\ctg M) \; | \; z(i) \in L_i, i=0,1\}.$$

We define a one-form on $\Omega(L_0,L_1)$ by:
$$\nu_z(V)=\int_0^1 \omega_M(z'(t),V(t)) \; dt.$$
The zeros of $\nu$ are the constant maps, that is, the intersection
points of $L_0$ and~$L_1$.

The integral of $\nu$ on a loop involves the one-form $u$, as in the
exact case, but also the monotonicity constant $k_L$ of $L$:

\begin{prop}
\label{prop:fctlacet}
  Let $\gamma: \cerc \rightarrow \Omega(L_0,L_1)$ be a loop. Then
\begin{eqnarray}
\int_\gamma \nu 
& = & \lambda_M (\gamma_0) - \lambda_M \left({\varphi_1^{-1}(\gamma_1)}\right) 
       - u ( \gamma_0 ) \label{eq:1}\\
& = & k_L \left({ \mu_{L_0} (\gamma_0) - \mu_{L_1} (\gamma_1) }\right) 
       - u ( \gamma_0 ) \label{eq:2}
\end{eqnarray}
denoting $\gamma_i= \gamma(\cerc \times \{i\})$ for $i=0,1$.
\end{prop}

\begin{proof}
Considering the loop $\gamma$ as a map 
$\gamma : \cerc \times [0,1] \rightarrow \ctg M$, we get
$$\int_\gamma \nu = 
\int_\cerc \nu \left( \frac{\partial \gamma}{\partial s} \right) ds
=\int_\cerc \int_0^1 \omega_M \left(\frac{\partial \gamma}{\partial t},
\frac{\partial \gamma}{\partial s} \right)  \;dt \;ds
= - \int_{\gamma(\cerc \times [0,1])} \omega_M.$$
Then, by the Stokes formula,
\begin{eqnarray*}
  \int_\gamma \nu &=& \int_{\gamma(\cerc \times \{0\})} \lambda_M - 
 \int_{\gamma(\cerc \times \{1\})} \lambda_M \\
                  &=& \int_{\gamma(\cerc \times \{0\})} \lambda_M - 
\int_{\varphi_1^{-1}(\gamma(\cerc \times \{1\}))} \varphi_1^\star \lambda_M \\
                  &=& \int_{\gamma(\cerc \times \{0\})} \lambda_M -
 \int_{\varphi_1^{-1}(\gamma(\cerc \times \{1\}))} \lambda_M
- \int_{\varphi_1^{-1}(\gamma(\cerc \times \{1\}))} 
\left(\varphi_1^\star \lambda_M - \lambda_M \right).
\end{eqnarray*}
As $\varphi_1^\star \lambda_M - \lambda_M$ is a closed one-form in the 
cohomology class $u= \Cal(\varphi_t)$, the third term is equal to: 
$$u \left( \varphi_1^{-1}(\gamma(\cerc \times \{1\})) \right) 
= u \left( \gamma(\cerc \times \{0\}) \right).$$
So that
$$ \int_\gamma \nu = \lambda_M \left(\gamma_0 \right)
       - \lambda_M \left( \varphi_1^{-1}(\gamma_1) \right) 
       - u \left( \gamma_0 \right).$$
We now use the monotonicity of $L$ to write:
\begin{eqnarray*}
\int_\gamma \nu &=& k_L \left({ \mu_L \left(\gamma_0 \right)
       - \mu_L \left( \varphi_1^{-1}(\gamma_1) \right) 
     }\right) - u \left( \gamma_0 \right)\\
                &=& k_L \left({ \mu_{L_0} \left(\gamma_0 \right)
       - \mu_{L_1} \left(\gamma_1 \right) 
     }\right) - u \left( \gamma_0 \right)
\end{eqnarray*}
since $\varphi_1$ is a symplectic isotopy.
\end{proof}

\begin{cor}
The action one-form $\nu$ is closed.
\end{cor}

\begin{proof}
 The formula (\ref{eq:1}) proves that $\int_\gamma \nu$ depends only
 on the homotopy class of $\gamma$ in $\Omega (L_0,L_1)$.
\end{proof}

\begin{rem}
\label{rem:fonctionelle} We could also lift $\nu$ to 
$\Omega(\bL_0,\bL_1)$ (as in the exact case, see~\cite{Mihai}) but the 
one-form is not necessarily exact on this space.
Nevertheless, we will be able to carry out the construction of the
complex without needing a primitive of~$\nu$.
\end{rem}

\subsubsection*{Alternative setting}

We can also define a one-form on $\Omega(L,L)$. This setting will be
useful in the proof of Hamiltonian invariance
(Section~\ref{sec:invHam}).

Let $X^{\alpha+dH_t}_t$ be the symplectic dual of $\alpha + dH_t$,
defined by 
$\omega_M( \: \cdot \: , X^{\alpha+dH_t}_t) = (\alpha + dH_t)(\cdot)$. 
Denote by $(\varphi_t)$ the isotopy spanned by $X^{\alpha+dH_t}_t$.

We can define a one-form $\hnu$ on $\Omega (L,L)$ by:
$$\hnu_z(V)=\int_0^1 \omega_M(z'(t),V(t))+(\alpha+dH_t)(V(t)) \; dt.$$
The zeros of $\hnu$ are the flow trajectories beginning on $L$ (at
time $0$) and ending on $L$ (at time $1$).

If $\gamma: \cerc \rightarrow \Omega(L,L)$ is a loop in $\Omega(L,L)$,
we have as in Proposition~\ref{prop:fctlacet}:
$$ \int_\gamma \hnu = - \int_{\cerc \times [0,1]} \gamma^\star \omega_M 
+ \int_\cerc \int_0^1 (\alpha+dH_t) \left(\frac{\partial \gamma}{\partial s} 
\right)  \;dt \;ds $$
with
\begin{eqnarray*}
 \int_\cerc \int_0^1 (\alpha+dH_t) \left(\frac{\partial \gamma}{\partial s} 
\right)  \;dt \;ds 
&=&  \int_\cerc \int_0^1 \alpha \left(\frac{\partial \gamma}{\partial s} 
\right)  \;dt \;ds  
= \int_0^1 \int_{\gamma(\cdot,t)} \alpha  \; dt \\
&=& \int_{\gamma(\cdot,0)} \alpha 
= u \left( \gamma_0 \right)
\end{eqnarray*}
since $ \int_{\gamma(\cdot,t)} \alpha$ does not depend on $t$. Thus,
we have

\begin{eqnarray}
\int_\gamma \hnu =
&=& \lambda_M(\gamma_0) - \lambda_M(\gamma_1) + 
u \left( \gamma_0 \right) \label{eq:alt1}\\
&=& k_L \left( \mu_L(\gamma_0) - \mu_L(\gamma_1) \right) + 
u \left( \gamma_0 \right). \label{eq:alt2}
\end{eqnarray}

\begin{rem}
Note that these two settings are equivalent: if $\nu_-$ is the
one-form defined on $\Omega(L_0,\varphi^{-1}(L_0))$ with the
symplectic isotopy $(\varphi_t^{-1})$, then the map 
$\Gamma(z)= \varphi_t^{-1}(z)$ is a $1$ to $1$ correspondence between
$\Omega(L,L)$ and $\Omega(L_0,\varphi^{-1}(L_0))$ and we have 
$$ \Gamma^\star \nu_- = \hnu.$$
(The fact that the correspondence uses $\varphi_t^{-1}$ instead of
$\varphi_t$ explains the difference of signs between the relations
(\ref{eq:1}) and (\ref{eq:2}) on the one hand, and (\ref{eq:alt1}) and
(\ref{eq:alt2}) on the other hand.)
\end{rem}

\subsection{The gradient}
\label{sec:grad}

Let $(J_t)$ be a family of almost complex structures on $\ctg M$ that
are compatible with $\omega_M$ and $(g_t)$ be the family of associated 
Riemannian metrics on $\ctg M$.

We consider the trajectories of the opposite of the gradient of the
one-form~$\nu$ with respect to the induced metric on $\Omega(L_0,L_1)$.  
These are, as maps of two variables, solutions of the Cauchy-Riemann 
equation.

We define for a solution $v$ of the Cauchy-Riemann equation its energy 
$$E(v) = \int_{\rr \times [0,1]} \left \|\frac{\partial v}{\partial s}
 \right \|^2 \;dt \;ds.$$
Denote then by $\cM(L_0,L_1)$ the space of trajectories of finite
energy: 
$$\cM(L_0,L_1)= \left \{
v \in \mathcal{C}^\infty(\rr \times [0,1],\ctg M) \left | 
  \begin{array}{l}
\dfrac{\partial v}{\partial s} + J_{t}(v) \dfrac{\partial v}{\partial t} = 0 \\[8pt]
v(s,0) \in L_0 \text{ and } v(s,1) \in L_1 \\[8pt]
E(v) < \infty
  \end{array}
\right.
\right \}.$$

In particular, if $v \in \cM(L_0,L_1)$, then as in 
Proposition~\ref{prop:fctlacet}:
$$E(v) = \int_{\rr \times [0,1]} v^\star \omega_M 
= \int_{\rr \times [0,1]} \omega_M \left(\frac{\partial v}{\partial s},
\frac{\partial v}{\partial t} \right) \;dt \;ds
= - \int_v \nu.$$

Denote also for $x,y \in L_0 \cap L_1$:
$$\cM(x,y)= \left \{
v \in \mathcal{C}^\infty(\rr \times [0,1],\ctg M) \left | 
  \begin{array}{l}
\dfrac{\partial v}{\partial s} + J_{t}(v) \dfrac{\partial v}{\partial t} = 0 \\[8pt]
\displaystyle \lim_{s \rightarrow - \infty} v(s,\cdot) = x \\[8pt]
\displaystyle \lim_{s \rightarrow + \infty} v(s,\cdot) = y
  \end{array}
\right.
\right \}$$
and
$$\cM^*(x,y)= \left \{ 
  \begin{array}{ll}
\cM(x,y)     & \text{ for } x \neq y \\
\cM(x,x) \bs \{x\} & \text{ for } x = y.
  \end{array}
\right.
$$

From \cite{MR965228} and \cite{MR1223659}, we have
\begin{thm}
\label{thm:reunion}
 $$ \cM(L_0,L_1)= \bigcup_{x,y \in L_0 \cap L_1} \cM(x,y).$$
\end{thm}

Let $(\tJ_{t})$ be a family of almost complex structures on $\ctg \tM$ 
obtained by lifting the family $(J_{t})$ and let $\tcM(\bL_{0},\bL_{1})$, 
$\tcM(\tx,\ty)$, and $\tcM^{*}(\tx,\ty)$ be the spaces of solutions
in $\ctg \tM$ defined in a similar way as in $\ctg M$. Then these
spaces also satisfy Theorem~\ref{thm:reunion} and $\tp$ maps
$\tcM(\bL_{0},\bL_{1})$ onto $\cM(L_0,L_1)$, so that a solution and
its image by $\tp$ have the same energy.\\ 

In the alternative setting, we consider the trajectories of the
opposite of the gradient of the one-form $\hnu$ with respect to the
metric defined on $\Omega(L,L)$ by a family of compatible almost
complex structures $(\hJ_t)$.

Denote
$$\hcM(L,L) = \left \{
v \in \mathcal{C}^\infty(\rr \times [0,1],\ctg M)\left | 
  \begin{array}{l}
\dfrac{\partial \hv}{\partial s} + 
\hJ_{t}(\hv) \left ( \dfrac{\partial \hv}{\partial t} 
- X_{t}^{\alpha + dH_{t}(\hv)} \right ) = 0 \\[8pt]
v(s,0) \in L \text{ and } v(s,1) \in L \\[8pt]
E(v) < \infty
  \end{array}
\right.
\right \},$$
with the energy defined by the same formula as above.

If $\hv \in \cM(L,L)$, then again
\begin{eqnarray*}
E(\widehat{v})
&=& \int_{\rr \times [0,1]} \left \| {\frac{\partial \hv}{\partial s}} \right \|^2  \;dt \;ds
= \int_{\rr \times [0,1]} \omega_M \left(\frac{\partial \hv}{\partial s},
\hJ_t(\hv) \frac{\partial \hv}{\partial s} \right) \;dt \;ds\\
&=& \int_{\rr \times [0,1]} \omega_M \left(\frac{\partial \hv}{\partial s},
\frac{\partial \hv}{\partial t}- X^{\alpha+dH_t}(\hv) \right) \;dt \;ds\\
&=& \int_{\rr \times [0,1]} \omega_M \left(\frac{\partial \hv}{\partial s},
\frac{\partial \hv}{\partial t} \right) \;dt \;ds
- \int_{\rr \times [0,1]} (\alpha + dH_t) \left(\frac{\partial \hv}{\partial s} \right) \;dt \;ds\\
&=& - \int_{\hv} \hnu 
\end{eqnarray*}
and we define analogously the space of solutions $\hcM(x,y)$ and 
$\hcM^*(x,y)$.

\begin{rem}
\label{rem:correspondance}
There is also a correspondence between the two settings for the 
gradient trajectories. For all $\hv \in \hcM(L,L)$, we can associate 
the map $v$ defined by  
$$v(s,t) = \varphi_t^{-1}(\hv(s,t)).$$
Let $(J_t)$ and $(\hJ_t)$ be two families of compatible almost complex 
structures on $\ctg M$ such that 
$$\hJ_t = (\varphi_t)_\star J_t (\varphi_t^{-1})_\star.$$
Then 
$$\frac{\partial v}{\partial s} + J_{t}(v) \frac{\partial v}{\partial t} =
(\varphi_t^{-1})_\star \left[ \frac{\partial \hv}{\partial s} + 
\hJ_{t}(\hv) \left ( \frac{\partial \hv}{\partial t} 
- X_{t}^{\alpha + dH_{t}(\hv)} \right ) \right],$$
and 
$$E(\hv)=E(v)$$
so that the map $\hv \mapsto v$ defines a bijection between  
$\hcM(L,L)$ and $\cM(L_0,\varphi_1^{-1}(L_0))$.
\end{rem}

\subsection{Transversality and compactness}
\label{sec:transv&compa}

In order to define our Floer-type homology, let us check now the
transversality and compactness requirements.

\subsubsection{Transversality}

We have a classical transversality result of Floer theory:

\begin{thm}
\label{thm:transv}
  Assume that $L_0$ et $L_1$ are transverse.\\
  Then for a generic choice of $J_t$, the spaces $\mathcal{M}(x,y)$ 
  are manifolds of finite dimension, of local dimension at 
  $v \in \mathcal{M}(x,y)$ the Maslov-Viterbo index
  (see~\cite{MR926533}) of $v$.\\  
  The same result is true for $\widetilde{\mathcal{M}}(\tx,\ty)$ and  
  the map $\tp$ induces a diffeomorphism 
  $$\tp : \widetilde{\mathcal{M}}(\tx,\ty) \rightarrow \mathcal{M}(x,y)$$
  for $\tp(\tx)=x$ and $\tp(\ty)=y$.
\end{thm}

\begin{proof}
  It is a transversality result analogous to \cite[Theorem
  3.12]{Mihai} and it can be proved as in~\cite{MR948771} (see
  also~\cite{MR1223659}).
\end{proof}

\begin{rem}
\label{rem:transversalite}
  In the case of the alternative setting, a one-form $\alpha+dH_t$
  being  given, there exists a generic Hamiltonian $h_t$ (with compact
  support) such that, if $(\psi_t)$ is the symplectic isotopy spanned
  by $\alpha+dH_t+dh_t$, then $L$ and $\psi_1(L)$ are transverse.

  Using the correspondence \ref{rem:correspondance}, we are then able
  to deduce from Theorem~\ref{thm:transv} that the spaces 
  $\widehat{\mathcal{M}}(x,y)$ are submanifolds for a generic choice  
  of family of compatible almost complex structure.
\end{rem}

\subsubsection{Compactness}

Let $x$ and $y$ be two intersection points of $L_0$ and $L_1$ and let  
$A > 0$. Denote by 
$$\cM^*_A(x,y)=\{v \in \cM^*(x,y) \;|\; E(v) \leq A\}$$
the space of solutions of finite energy between $x$ and $y$.

The translation in the $s$ variable (defined by $(\sigma \cdot v)(s,t) 
= v(\sigma+s,t)$) induces a free action of $\rr$ on $\cM^*(x,y)$. Let
$\mathcal{L}(x,y)$ denote the quotient $\cM^*(x,y)/\rr$.

Let us recall the result of convergence modulo bubbling of a sequence
of elements of $\cM^*_A(x,y)$ (as stated 
in \cite[Proposition 3.7]{MR1223659}): if $(v_n)$ is a sequence of
elements of $\cM^*_A(x,y)$ with a fixed index equal to $I$, then there 
exists a subsequence converging (modulo translations, i.e. in the
quotient $\mathcal{L}(x,y)$) to a ``cusp'' curve  $(\uv,\uw,\uu)$
(where $\uv$ is a finite collection of solutions 
$v^i \in \cM^*_A(z_i,z_{i+1})$, $\uw$ is a finite collection of 
$J$-holomorphic disks $w^j$ and $\uu$ is a finite collection of
$J$-holomorphic spheres $u^k$) such that
\begin{eqnarray*}
  \sum_i \int (v^i)^\star \omega + \sum_j \int (w^j)^\star \omega + 
\sum_k \int (u^k)^\star \omega & \leq & A \\
\sum_i \mu(v^i) + \sum_j \mu(w^j) + \sum_k 2 c_1(u^k) & = & \mu_0.
\end{eqnarray*}

In our case, the ambient symplectic manifold is the cotangent bundle
of the manifold $M$ so that no bubbling of $J$-holomorphic spheres 
occurs. 

\begin{prop}
\label{prop:compa}
  Let $L$ be a monotone Lagrangian manifold of Maslov number 
  $N_L \geq 3$.  
  Then:
  \begin{description}
  \item[(i)] For any sequence of elements $(v_{n})$ in the
    one-dimensional component of $\cM^*_A(x,y)$, there exists a
    sequence $(\sigma_{n})$ of real numbers, such that a subsequence
    of $(\sigma_{n} \cdot v_{n})$ converges in $\cM^*_A(x,y)$;
\item[(ii)] if $(v_n)$ is a sequence of elements of $\cM^*_A(x,z)$ of  
  index $2$, then
  \begin{itemize}
  \item either there exists a sequence $(\sigma_{n})$ of real numbers 
    such that a subsequence of $(\sigma_{n} \cdot v_{n})$ converges to
    a solution $v$ in $\cM^*_A(x,z)$;
  \item or there exists a pair of sequences
    $((\sigma^1_{n}),(\sigma^2_{n}))$ of real numbers and a pair of
    solutions $(v^1,v^2) \in \cM^*_A(x,y) \times \cM^*_A(y,z)$, for
    some intersection point $y$, such that, for every $i \in \{1;2\}$,
    a subsequence of $(\sigma^i_{n} \cdot v_{n})$ converges to $v^i$
    (in this case one says that $(v_n)$ converges to the broken orbit
    $(v^1,v^2))$.
\end{itemize}
\end{description}
\end{prop}

\begin{rem}
\label{rem:NL=2}
  Proposition~\ref{prop:compa} will be used in
  Section~\ref{sec:Floerdiff} to prove the compactness of the space of 
  trajectories. Note that statement (i) in
  Proposition~\ref{prop:compa} also holds in the case $N_L=2$. We will
  deal with the convergence of a sequence of elements of
  $\cM^*_A(x,z)$ of index~$2$ in the case $N_L=2$ in the proof of
  Lemma~\ref{lem:dcarre}.
\end{rem}

\noindent
\textit {Proof of Proposition~\ref{prop:compa}.} 
\begin{description}
  \item[(i)] If $(v_n)$ is a sequence of elements of $\cM^*_A(x,y)$ of 
    index $1$, then there exists a subsequence that converges to a 
    << cusp >> curve $(\uv,\uw,\emptyset)$ such that 
    \begin{eqnarray}
      \sum_i \int (v^i)^\star \omega + \sum_j \int (w^j)^\star \omega 
      & \leq & A 
      \label{eqn:airebub} \\
      \sum_i \mu(v^i) + \sum_j \mu(w^j) & = & 1. \label{eqn:Maslovbub} 
    \end{eqnarray}
    Since the area of a $J$-holomorphic disk $w^j$ is non-negative,
    the monotonicity assumption on $L$ (the monotonicity on the disks
    is sufficient here) gives $\mu(w^j) \geq 3$.  
    Because of (\ref{eqn:Maslovbub}), there is no bubbling of
    $J$-holomorphic disk ($\uw = \emptyset$). 
    Moreover, the dimension of $\cM^*(x,y)$ is at least $1$ (because
    of the free action of $\rr$), so that the Maslov class of a
    solution $v^{i}$ is at least~$1$. Hence, the collection $\uv$ can
    only contain one element  that belongs to $v \in \cM^*_A(x,y)$.

  \item[(ii)] As in (i), no bubbling of $J$-holomorphic disks can
    occur. As a consequence, there is a subsequence converging to a
    broken orbit $\uv$ such that 
    $$\sum_i \mu(v^i) = 2$$
    and hence $\uv$ admits at most two components. \hfill \qedsymbol
\end{description}

For the construction of the complex we need a homotopy lemma:

\begin{lem}
\label{lem:hom}
Let $(v_n)$ be a sequence of elements of $\cM^*_A(x,y)$ of index
either~$1$ or $2$ having a subsequence converging either to  
$\uv = v^1 \in \cM^*_A(x,y)$ or to $\uv = \{v^1,v^2\}$ with 
$(v^1,v^2) \in \cM^*_A(x,y) \times \cM^*_A(y,z)$.

Let $\gamma_n: [-\infty; + \infty] \rightarrow L_0$ be the path
defined by $\gamma_n(s)=v_n(s,0)$ (extended at $s=-\infty$ by $x$ and
at $s=+\infty$ by $y$). 
Let $\gamma^i:[-\infty; + \infty] \rightarrow L_0$ be the paths
defined analogously for the $v^i$.

Then, for $n$ large enough, $\gamma_n$ is homotopic to either
$\gamma^1$ (when $\uv = v^1$) as a path from $x$ to $y$, or to the
concatenation of paths $\gamma^1 \star \gamma^2$ (when 
$\uv = \{v^1,v^2\}$) as a path from $x$ to $z$. 
\end{lem}

\begin{proof}
The proof is similar to \cite[Lemma 3.16]{Mihai} if we add the
assumption on the index ensuring that no bubbling of $J$-holomorphic
disks can occur. 
\end{proof}

\subsection{The differential of the Floer complex}
\label{sec:Floerdiff}

Let $x$ and $y$ be two intersection points of $L_0$ and $L_1$. In this
section, we define an incidence number $[x,y]$. 

Let $\mathcal{L}^0(x,y)$ be the zero-dimensional component of 
$\mathcal{L}(x,y)$.
For all $z \in L_0 \cap L_1$, fix a lift $\tz \in \ctg \tM$.
For $g \in \pi_1(M)$, denote by $\cL^0_g(x,y) \subset \cL^0(x,y)$ the
subset of trajectories that lift to $\widetilde{\cL}(\tx^g, \ty)$
(with the same notation for the action of $\pi_1(M)$ as in
Lemma~\ref{lem:rev}). 

Let us state and prove a lemma that will replace 
\cite[Lemma 3.16]{Mihai} in our construction. 

\begin{lem} 
\label{lem:L0fini}
Assume that $N_L \geq 2$. For all $x$, $y$ in $L_0 \cap L_1$ and all
$g$ in $\pi_1(M)$, the set $\cL^0_g(x,y)$ is finite.

If $n_g$ denotes the cardinal modulo $\zz/2$, the number $\sum n_g g$
belongs to the Novikov ring $\Lambda_{-u}$. 
\end{lem}

\begin{proof}
The elements of $\cL^0_g(x,y)$ are classes of solutions $v$ which
belong to the one-dimensional component of $\cM^*(x,y)$. 
We prove that these solutions all have the same energy. For that
purpose, we prove that two solutions from $x$ to $y$, which have the
same index, and which can be both lifted to trajectories from $\tx^g$
to $\ty$, have the same energy. 

We consider a solution 
$$v:[-\infty,+\infty]  \times [0,1] \longrightarrow \ctg M$$ 
in $\cM(x,y)$ as a path in $\Omega(L_0,L_1)$ from $x$ to $y$. 
If $v$ is such a path, let 
$$\bv:[-\infty,+\infty]  \times [0,1] \longrightarrow \ctg M$$ 
be the ``inverse'' path defined by  
$$\bv(s,t) = v(-s,t).$$

Let $v_1$ and $v_2$ be two elements of $\cM(x,y)$ satisfying 
$\mu(v_1)=\mu(v_2)$.
If we denote by $\gamma = v_2 \# \overline{v}_1$ the concatenation of
the paths $v_2$ and $\overline{v}_1$ (in this order), then $\gamma$ is
a loop in $\Omega(L_0,L_1)$ based in $x$ (see Figure
\ref{fig:lacetdiff}). Note that here, we use the notation
$\Omega(L_0,L_1)$ for the space of paths from $L_0$ to $L_1$ which
are piecewise smooth instead of just smooth as in
Section~\ref{sec:fctaction}. 

\begin{figure}[htbp]
  \begin{center}
   \psfrag{v2}{$v_2$}
   \psfrag{v1}{$\overline{v}_1$}
   \psfrag{x}{$x$}
   \psfrag{y}{$y$}
   \psfrag{L0}{$L_0$}
   \psfrag{L1}{$L_1$}
   \includegraphics[height=3cm]{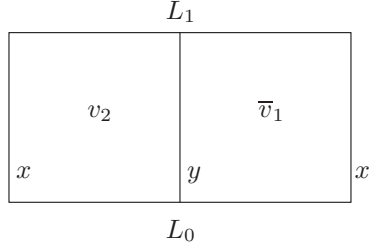}
    \caption{the loop $\gamma$} 
    \label{fig:lacetdiff}
  \end{center}
\end{figure}
Then
\begin{eqnarray*}
  \int_\gamma \nu &=& - \int \gamma^\star \omega_M \\
   &=& - \int v_2^\star \omega_M + \int v_1^\star \omega_M \\
   &=& E(v_1) - E(v_2).
\end{eqnarray*}
and thanks to the monotonicity on the loops, by
Proposition~\ref{prop:fctlacet}, 
$$\int_\gamma \nu = 
k_L (\mu_{L_0}(\gamma_0) - \mu_{L_1}(\gamma_1)) - u(\gamma_0)$$
with 
$\mu_{L_0}(\gamma_0) - \mu_{L_1}(\gamma_1) = \mu(v_2) - \mu(v_1) = 0$.

Moreover, if the lifts of $v_1$ and $v_2$ are trajectories from
$\tx^g$ to $\ty$, $\gamma_0$ can be lifted to a loop based in
$\tx^g$. Therefore, $\gamma_0$ is homotopic to the constant loop in
$\ctg M$ and this implies that $u(\gamma_0)=0$ and 
$\int_\gamma \nu = 0$. We have thus proved that $v_1$ and $v_2$ have
the same energy.\\

Assume that $N_L \geq 2$. We can then apply
Proposition~\ref{prop:compa} and Remark~\ref{rem:NL=2}: a  sequence of
solutions $(v_n)$ between $x$ and $y$ of Maslov index $1$ has a
subsequence converging to a solution of $\cM^*(x,y)$. Thanks to
Lemma~\ref{lem:hom}, this solution can be lifted to a trajectory from
$\tx^g$ to $\ty$. This means that the space $\cL^0_g(x,y)$ is a
compact space of dimension zero and hence it has only a finite number
of elements.\\

For the second part of the lemma, it is enough to show that for 
$C > 0$, the set $$\bigcup_{-u(g) \leq C} \cL^0_g(x,y)$$ is compact
(so that it is finite).
A sequence $(\underline{v}_n)$ in this space can be lifted to a
sequence $(v_n)$ in the one-dimensional component of $\cM^*(x,y)$. But
if $v_1$ and $v_2$ are two solutions from $x$ to $y$ with same Maslov
class, we have:
$$E(v_1) - E(v_2) = k_L (\mu(v_2) - \mu(v_1)) - u(\gamma_0) 
= - u(\gamma_0),$$
where $\gamma$ denotes the concatenation $v_2 \# \overline{v}_1$ as
above.

If $v_1$ can be lifted to a trajectory from $\tx^{g_1}$ to $\ty$ and
$v_2$ to a trajectory from $\tx^{g_2}$ to $\ty$, $\gamma$ can be lifted
as a path from $\tx^{g_2}$ to $\tx^{g_1}$ so that $u(\gamma_0) =
u(g_2^{-1} g_1)$ and  
$$E(v_1) - E(v_2) = u(g_2)-u(g_1).$$
As a consequence, if we consider a sequence of solutions $(v_n)$, each 
$v_n$ being lifted as a trajectory from $\tx^{g_n}$ to $\ty$ with 
$- u(g_n) \leq C$, then:
$$E(v_n) = E(v_0) - u(g_n) + u(g_0) \leq C + E(v_0) + u(g_0).$$
The energy of the elements of this sequence is bounded and we can
apply Proposition~\ref{prop:compa}: $(v_n)$ has a converging
subsequence and the limit of this subsequence can be lifted to a path
from $\tx^{g_{\infty}}$ to $\ty$ which satisfy: 
$$- u(g_\infty) = E(v_\infty) - E(v_0) + u(g_0) \leq C.$$
This means that $(\underline{v_n})$ has a converging subsequence in 
$\displaystyle \bigcup_{-u(g) \leq C} \cL^0_g(x,y)$ which is 
therefore compact. \end{proof}

We can now define the incidence number:
$$[x,y] = \sum_{g \in \pi_1(M)} n_g(x,y) g,$$
where $n_g(x,y)$ is the cardinality of $\cL^0_g(x,y)$.
We define the complex $C_\star (\bL_0,\bL_1,J_t)$ as the
$\Lambda_u$-vector space spanned by the intersection points of $L_0$
and $L_1$ endowed with the differential: 
$$\partial x = \sum_{g \in \pi_1(M), y \in  L_0 \cap L_1} n_g(x,y) g y.$$

\begin{lem}
\label{lem:dcarre}
  If $N_L  \geq 2$, $\partial \circ \partial = 0$. 
\end{lem}

\begin{proof}
In order to prove the relation $\partial^2 = 0$, one has to prove that
for all $g \in \pi_1(M)$ and all $x,z \in L_0 \cap L_1$, we have: 
\begin{equation}
  \label{eq:d2}
  \sum_{y \in L_0 \cup L_1, g', g'' \in \pi_1(M), g'g''=g} n_{g'}(x,y) 
n_{g''}(y,z) = 0.
\end{equation}

When $N_L \geq 3$, this is, using Proposition~\ref{prop:compa} as in
\cite[Lemma 3.18]{Mihai}, a consequence of the compactification of the
one-dimensional component of $\cL_g(x,y)$ with broken trajectories
(see figure \ref{fig:compacite}). 
This compactification is a compact $1$-dimensional manifold whose
boundary is 
$$\bigcup_{y \in L_0 \cup L_1, g', g'' \in \pi_1(M), g'g''=g} 
\cL^0_{g'}(x,y) \times \cL^0_{g''}(y,z).$$

\begin{figure}[htbp]
  \begin{center}
   \psfrag{x}{$x$}
   \psfrag{y}{$y$}
   \psfrag{z}{$z$}
   \includegraphics[height=5cm]{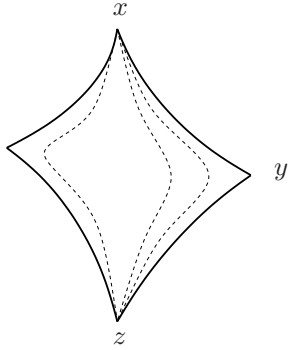}
   \caption{The compactification with broken trajectories} 
   \label{fig:compacite}
  \end{center}
\end{figure}

Let us consider now the case $N_L=2$. Oh noticed in \cite{MR1367384}
that his extension of the Floer complex to the monotone case is
possible under this assumption. This is also possible for the
Floer-Novikov complex.
Indeed, in the proof that the zero-dimensional component of 
$\cL_g(x,y)$ is compact, one only needs that $N_{L} \geq 2$ and it is
then possible to define the Floer differential.

The condition $N_L \geq 3$ is used to avoid bubbling in the
convergence of a sequence of solutions of index $2$ and prove that the
square of the differential is zero. Let us prove that this is also
true for $N_{L}=2$. 

The only sequences for which the bubbling of a $J$-holomorphic disk
can occur are sequences of solutions from an intersection point $x$ to
itself that have Maslov index $2$. Then the ``bubble'' also has Maslov
index $2$.

In this case, it is possible to compactify the one-dimensional
component of $\cL_g(x,x)$ by adding to the broken trajectories the
pairs formed by a constant trajectory and a $J$-holomorphic disk with
boundary either on $L_0$ or $L_1$ (this is similar to
\cite{MR1367384}). 

\begin{figure}[htbp]
  \begin{center}
   \psfrag{x}{$x$}
   \psfrag{DL0}{disk with boundary on $L_0$}
   \psfrag{DL1}{disk with boundary on $L_1$}
   \psfrag{Jstrip}{$J$ holomorphic strips}
   \includegraphics[height=6cm]{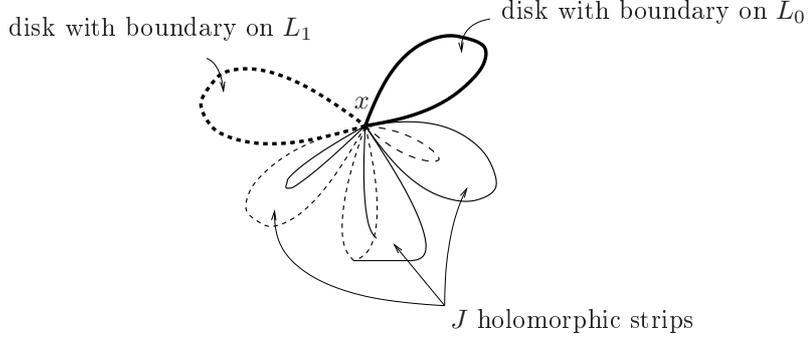} 
   \caption{The compactification with disks (the plain lines
     correspond to boundaries on $L_0$, the dashed lines to boundaries 
     on $L_1$)} 
    \label{fig:compacitebulle}
  \end{center}
\end{figure}

Note also that only a sequence of solutions that can be lifted to
paths from $\tx$ to $\tx$ can converge to a $J$-holomorphic disk, so
that this type of compactification is only needed for $\cL_e(x,x)$,
where $e$ is the identity element of $\pi_1(M)$. Hence (\ref{eq:d2})
holds also for $g \neq e$ with the previous type of compactification. 

When $g=e$, we have as in \cite{MR1367384}, 
$$ \sum_{y \in L_0 \cup L_1, g', g'' \in \pi_1(M), 
g'g''=g} n_{g'}(x,y) n_{g''}(y,z) = \Phi_{L_0}(x)+ \Phi_{L_1}(x)$$
where $\Phi_{L_i}(x)$ is the number (modulo $2$) of $J$-holomorphic
disks with Maslov index $2$ with boundary on $L_i$ and that pass
through the point $x$. Here we use that $\Phi_{L_i}(x)$ is preserved
under symplectic isotopies (in \cite{MR1367384}, Oh uses only
Hamiltonian isotopies but the proof is similar with symplectic
isotopies) to see that $\Phi_{L_0}(x)+ \Phi_{L_i}(x)= 0 \mod 2$. 
Therefore, $\partial \circ \partial= 0$ even in the case $N_L = 2$.
\end{proof}

\begin{rem}
 One can define the same way a complex 
$C_\star (\bL,\varphi_t, \widehat{J}_t)$ 
spanned by the zeros of the one-form $\hnu$ and define a differential
using the spaces $\widehat{\cM}(x,y)$. 
By the correspondence~\ref{rem:correspondance}, the
$\Lambda_u$-complexes  
$C_\star ( \bL,\varphi_t,  \widehat{J}_t)$ and 
$C_\star (\bL_0,\tvarphi_1\,^{-1}(\bL_0), J_t)$  are isomorphic.
\end{rem}

\subsection{Hamiltonian invariance}
\label{sec:invHam}

Denote $H_\star(\bL_0,\bL_1,J_t)$ the homology of the complex 
$C_\star(\bL_0,\bL_1,J_t)$ defined in Section~\ref{sec:Floerdiff}. 
We have assumed that $L_1 = \varphi_1(L_0)$ where the isotopy
$(\varphi_t)$ is supposed to be spanned by $\alpha + dH_t$ with
$\alpha$ a closed one-form on $M$ and $H$ a Hamiltonian with compact
support on $\ctg M \times [0,1]$.

We now prove that this homology does not depend on the generic choice
of the pair $(J_t, H_t)$. For that purpose, we will use the
``alternative'' setting and the complex 
$C_\star(\bL,\varphi_t,J_t)$. We denote by 
$H_\star (\bL,\varphi_t^{\alpha+dH_t},J_t)$ its homology.

\begin{thm}
\label{thm:invar-hamilt}
  For any two generic pairs $(H_t,J_t)$ and $(H'_t,J'_t)$, there
  exists an isomorphism 
  $$\Psi : H_\star(\bL,\varphi_t^{\alpha+dH_t},J_t) \longrightarrow 
  H_\star(\bL,\varphi_t^{\alpha+dH'_t},J'_t).$$
\end{thm}

\noindent
\textit{Proof.}
As in the proof of \cite{Mihai}, define a morphism of
$\Lambda_u$-complexes  
$$\Psi_\star : C_\star (\bL,\varphi_t^{\alpha+dH_t},J_t) 
\longrightarrow C_\star (\bL,\varphi_t^{\alpha+dH'_t},J'_t)$$ 
associated to a family of functions $H_{s,t}: \ctg M \rightarrow \rr$
and a family of compatible almost complex structures $J_{s,t}$
continuous in $(s,t) \in \rr^2$ and satisfying:
$$(H_{(s,t)},J_{(s,t)}) = \left \{ 
  \begin{array}{l}
(H_t,J_t) \text{ for } s < -R \\
(H'_t,J'_t) \text{ for } s > R.
  \end{array}
\right.$$

Consider the space 
$$\cM_{H_{(s,t)},J_{(s,t)}}(L) = \left \{
v : \rr \times [0,1] \rightarrow \ctg M \left | 
  \begin{array}{l}
\dfrac{\partial v}{\partial s} + J_{s,t} \left ( \dfrac{\partial v}{\partial t} 
- X_{s,t}^{\alpha + dH_{s,t}} \right ) \\[8pt]
v(s,i) \in L \text{ pour } i =0,1, s \in \rr \\[8pt]
E(v) < \infty
  \end{array}
\right.
\right \}.$$
An element $v$ of this space converges to a zero $x$ of the one-form
$\hnu$ when $s$ goes to $- \infty$ and to a zero $y$ of the one-form
$\hnu\,'$ (which corresponds to the Hamiltonian~$H'_t$) when $s$ goes
to $+ \infty$. 
As in Theorem \ref{thm:reunion}, we have: 
$$\cM_{H_{(s,t)},J_{(s,t)}}(L)  = 
\bigcup_{x,y} \cM_{H_{(s,t)},J_{(s,t)}}(x,y)$$
with 
$$\cM_{H_{(s,t)},J_{(s,t)}}(x,y) = \left \{
v : \rr \times [0,1] \rightarrow \ctg M \left | 
  \begin{array}{l}
\dfrac{\partial v}{\partial s} + J_{s,t} \left ( \dfrac{\partial v}{\partial t} 
- X_{s,t}^{\alpha + dH_{s,t}} \right ) \\[8pt]
v(s,i) \in L \text{ for } i =0,1, s \in \rr \\[8pt]
\displaystyle \lim_{s \rightarrow - \infty} v(s,t) = x(t) \\[8pt]
\displaystyle \lim_{s \rightarrow + \infty} v(s,t) = y(t) 
  \end{array}
\right.
\right \},$$
where $x$ (respectively $y$) are the zeros of the one-form $\hnu$ 
(respectively $\hnu\,'$).\\

We also have a transversality result for these spaces: for a generic
choice of the pair $(H_{s,t},J_{s,t})$, the spaces
$\cM_{H_{(s,t)},J_{(s,t)}}(x,y)$ are manifolds of local dimension
given by the Maslov index of a solution.\\ 

We also need a compactness result:

\begin{lem}
  For all $A > 0$, the zero-dimensional component of  
$$\cM_{H_{(s,t)},J_{(s,t)}}(x,y;A) = 
\left \{ { v \in \cM_{H_{(s,t)},J_{(s,t)}}(x,y) 
\; | \; E(v) \leq A } \right \}$$
is finite.
\end{lem} 

\begin{proof}
  The proof is standard since no bubbling occurs in dimension $0$ (see
  e.g. \cite[Lemma 3.22]{Mihai}).
\end{proof}

As before, we fix a lift $\tx$ in $\ctg \tM$ for every zero $x$ of the
one-form $\hnu$ and a lift $\ty$ for every zero $y$ of the one-form
$\hnu\,'$ (remember that the zeros of $\hnu$ are flow trajectories
beginning on $L$ and ending on $L$). Consider for all $g$ of
$\pi_1(M)$ and all zeros $x$ and $y$ of $\hnu$ and $\hnu\,'$
respectively, the space $\cM_{g,s}(x,y) \subset
\cM_{H_{(s,t)},J_{(s,t)}}(x,y)$ of solutions that can be lifted to
$\ctg \tM$ in paths from $\tx^g$ to $\ty$. We show that: 

\begin{prop}
\label{prop:invar-hamilt}
For any fixed index $I$, let  
$$\cM^I_{g,s}(x,y) = 
\cM_{g,s}(x,y) \cap \cM^I_{H_{(s,t)},J_{(s,t)}}(x,y)$$
be the space of elements of index $I$ in $\cM_{g,s}(x,y)$.
The energy of any element of $\cM^I_{g,s}(x,y)$ is bounded from above 
by some positive constant $A$, that is :
$$\cM^I_{g,s}(x,y) \subset \cM_{H_{(s,t)},J_{(s,t)}}(x,y;A).$$
\end{prop}

\begin{cor}
  The space $\cM^0_{g,s}(x,y)$ is finite.
\end{cor}

\noindent
\textit{Proof of Proposition~\ref{prop:invar-hamilt}.}
 Here again, we adapt the proof of \cite{Mihai} (see
 also~\cite{MR1308493}).  
 We do not have a primitive of $\hnu$ but it is possible to bound from
 above the difference of the energy of two solutions by a constant
 independent of the solutions.  
 
 Consider the norm defined by the compatible metric  
$\omega_M( \: \cdot \: ,J_{(s,t)} \cdot)$. The energy of a solution
$v$ of $\cM_{H_{(s,t)},J_{(s,t)}}(L)$ can be written:
\begin{eqnarray*}
  E(v) 
&=& \int_{\rr \times [0,1]} \left \| \frac{\partial v}{\partial s} \right  \|^2
\;dt \;ds\\
&=& \int_{\rr \times [0,1]} \omega_M \left( \frac{\partial v}{\partial s}, 
J_{s,t} \frac{\partial v}{\partial s} \right) \;dt \;ds \\
&=& \int_{\rr \times [0,1]} \omega_M \left( \frac{\partial v}{\partial s},
 \frac{\partial v}{\partial t} - X_{s,t}^{\alpha + dH_{s,t}} \right) \;dt \;ds  \\
&=& \int_v \omega_M - \int_{\rr \times [0,1]} (\alpha + dH_{s,t})\left( \frac{\partial v}{\partial s} \right) \;dt \;ds.
\end{eqnarray*}

Let $z_0 \in \Omega(L,L)$ and let $w$ be a fixed path in $\Omega(L,L)$ 
that joins $y$ to $z_0$.

We consider $v:[-\infty,+\infty]  \times [0,1]$ as a path in
$\Omega(L,L)$ from $x$ to $y$ and we use the concatenation $v \# w$ of
$v$ and $w$. 

\begin{figure}[htbp]
  \begin{center}
   \psfrag{v}{$v$}
   \psfrag{w}{$w$}
   \psfrag{x}{$x$}
   \psfrag{y}{$y$}
   \psfrag{z0}{$z_0$}
   \psfrag{L}{$L$}
   \includegraphics[height=3cm]{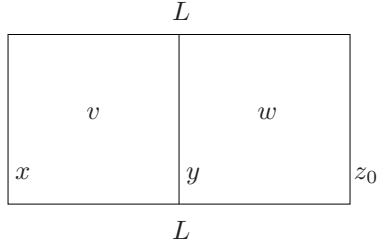}
    \caption{The path $v \# w$} 
    \label{fig:chemininvham}
  \end{center}
\end{figure}

Suppose that $w$ has been chosen in such a way that it can be lifted
to a path joining $\ty$ and $\tz_0$. 

We prove  
\begin{equation}
\label{eq:invar-hamilt}
\int_{v \# w} \hnu - \int_w \hnu\,' 
= -E(v) + \int_{\rr \times [0,1]} 
\frac{\partial H}{\partial s}(s,t,v) \;dt \;ds + C_1  
\end{equation}
where  
$$C_1  =  \int_{[0,1]} H_t(z_0) - H'_t(z_0) \; dt$$
is independant of $v$.\\

We have:
\begin{eqnarray*}
  \int_{v \# w} \hnu - \int_w \hnu\,' 
&=& 
- \int_{v \# w}  \omega_M 
+ \int_{\rr \times [0,1]} (\alpha+dH_t) \left(\frac{\partial v}{\partial s} 
\right)  \;dt \;ds \\
&&+ \int_{\rr \times [0,1]} (\alpha+dH_t) \left(\frac{\partial w}{\partial s} 
\right)  \;dt \;ds\\
&& + \int_{w}  \omega_M 
- \int_{\rr \times [0,1]} (\alpha+dH'_t) \left(\frac{\partial w}{\partial s} 
\right)  \;dt \;ds \\
&=& 
- \int_{v}  \omega_M + \int_{\rr \times [0,1]} \frac{\partial }{\partial s} \left( {H_t(v) + H_t(w) 
- H'_t(w) } \right) \;dt \;ds \\
&&+  \int_{\rr \times [0,1]} \alpha \left({\frac{\partial v}{\partial s}} \right),
\end{eqnarray*}
with
\begin{align*}
\int_{\rr \times [0,1]} \frac{\partial }{\partial s} &\left( {H_t(v) +
    H_t(w) - H'_t(w) } \right) \;dt \;ds = \\
&= \int_{[0,1]} H'_t(y) - H_t(x) \; dt +  \int_{[0,1]} H_t(z_0) - H'_t(z_0)  \; dt \\
&= \int_{\rr \times [0,1]} \frac{\partial }{\partial s} H_{s,t}(v)
\;dt \;ds + C_1 \\
&= \int_{\rr \times [0,1]} dH_{s,t} \left({\frac{\partial v}{\partial s}} \right) \;dt \;ds + \int_{\rr \times [0,1]} \frac{\partial H}{\partial s}(s,t,v) \;dt \;ds + C_1.
\end{align*}
Consequently 
\begin{eqnarray}
\int_{v \# w} \hnu - \int_w \hnu\,' 
&=& 
- \int_v \omega_M + \int_{\rr \times [0,1]} (\alpha + dH_{s,t}) \left( \frac{\partial v}{\partial s} \right) \;dt \;ds \nonumber \\
&& + \int_{\rr \times [0,1]} \frac{\partial H}{\partial s}(s,t,v) \;dt \;ds + C_1 \nonumber \\
\int_{v \# w} \hnu - \int_w \hnu\,' 
&=& -E(v) + \int_{\rr \times [0,1]} \frac{\partial H}{\partial s}(s,t,v) \;dt \;ds + C_1.
\end{eqnarray}

Now, let $v_1$ and $v_2$ be two elements of $\cM_{g,s}(x,y)$. From  
(\ref{eq:invar-hamilt}), we have:
\begin{eqnarray*}
E(v_1)-E(v_2)&=& - \int_{v_1 \# w} \hnu + \int_{v_2 \# w} \hnu \\
&& +  \int_{\rr \times [0,1]} \frac{\partial H}{\partial s}(s,t,v_1) \;dt \;ds
-  \int_{\rr \times [0,1]} \frac{\partial H}{\partial s}(s,t,v_2) \;dt \;ds.
\end{eqnarray*}
As 
$$\frac{\partial H}{\partial s}: \rr \times \ctg M 
\rightarrow \rr$$ has compact support, there exists a constant 
$C_2 \geq 0$ that does not depend on the $v_i$ ($i=1,2$) such that  
$$\int_{\rr \times [0,1]} \frac{\partial H}{\partial s}(s,t,v_1) \;dt \;ds
-  \int_{\rr \times [0,1]} \frac{\partial H}{\partial s}(s,t,v_2) \;dt \;ds \leq C_2.$$

Moreover, if $\gamma$ is a loop (based in $x$) of $\cM(L,L)$ obtained
by concatenation of the paths $v_2 \# w$ and $\overline{v_1 \# w}$,
then by the monotonicity condition (and Formula~(\ref{eq:alt2}))
\begin{eqnarray*}
 -  \int_{v_1 \# w} \hnu + \int_{v_2 \# w} \hnu
&=&  \int_{\gamma} \hnu \\
&=&  k_L \left( \mu_L(\gamma_0) - \mu_L(\gamma_1) \right) + 
u \left( \gamma_0 \right)
\end{eqnarray*}
with $ \mu_L(\gamma_0) - \mu_L(\gamma_1) = \mu(v_2)-\mu(v_1) = 0$.

The paths $v_1$ and $v_2$ belong to $\cM_{g,s}(x,y)$, so that $\gamma$ can
be lifted to a path from $\tx^g$ to $\tx^g$. The path $\gamma_0$ is
then homotopic to the constant path and we obtain the inequality: 
$$E(v_1)-E(v_2) \leq C_2.$$

This proves that if we fix an element $v_0$ in $\cM_{g,s}(x,y)$, 
then for all 
$v$ in $\cM_{g,s}(x,y)$, $E(v) \leq E(v_0) + K$ and consequently  
 $\cM_{g,s}(x,y)$ is contained in $\cM_{H_{(s,t)},J_{(s,t)}}(x,y;A)$ 
for some positive constant $A$. 
This proves Proposition~\ref{prop:invar-hamilt}. \hfill \qedsymbol\\

The space $\cM^0_{g,s}(x,y)$ is thus finite and we can define the
morphism of complexes  
$$\Psi_\star : C_\star (\bL,\varphi_t^{\alpha+dH_t},J_t) 
\longrightarrow C_\star (\bL,\varphi_t^{\alpha+dH'_t},J'_t)$$
by 
$$\Psi_\star (x) = \sum_{g \in \pi_1(M), y} m_g(x,y) g y$$
where $m_g(x,y)$ is the cardinality of $\cM^0_{g,s}(x,y)$ modulo $2$. 

In order to check that the coefficients belong to $\Lambda_u$, we use
the computations in the proof of \ref{prop:invar-hamilt} but this time
with $v_1$ in $\cM_{g_1,s}(x,y)$ and $v_2$ in $\cM_{g_2,s}(x,y)$ for
two elements $g_1$ and $g_2$ of $\pi_1(M)$. 
The loop $\gamma$ can then be lifted to a path from $\tx^{g_2}$ to
$\tx^{g_1}$ so that 
$u(\gamma_0)= u(g_2^{-1}g_1) = - u(g_2) + u(g_1)$ and
$$E(v_1)-E(v_2) \leq u(g_1) - u(g_2) + C_2.$$
If $v_0$ is a fixed element of $\cM_{g_0,s}(x,y)$, we have for all $v$
of $\cM_{g,s}(x,y)$ with $u(g) < C$, 
$$E(v) \leq E(v_0) + C - u(g_0) + C_2,$$
and this implies that 
$\displaystyle \bigcup_{u(g)<C} \cM^0_{g,s}(x,y)$ is contained in
$\cM^0_{H_{(s,t)},J_{(s,t)}}(x,y;A)$ for some positive constant $A$,
so that this union is finite. \\

We use the usual methods of Floer theory to finish the proof of the
theorem:  
\begin{itemize}
\item The fact that $\Psi_\star$ commutes with the differentials comes
  from the study of the compactification with the help of broken
  trajectories of the \hbox{$1$-dimensionnal} component of
  $\cM_{H_{(s,t)},J_{(s,t)}}(x,y)$.  
\item The map $\Psi_\star$ induces an isomorphism in homology: to
  prove this, it is enough to consider the morphism defined
  analogously between 
  $$C_\star (\bL,\varphi_t^{\alpha+dH'_t},J'_t) 
  \mbox{ and } C_\star (\bL,\varphi_t^{\alpha+dH_t},J_t)$$
  and to show that the composition of these morphisms are homotopic to
  the identity. \hfill \qedsymbol\\ 
\end{itemize}

\begin{notn}
By Theorem \ref{thm:invar-hamilt}, the homology of the complex  
$C_\star (\bL,\varphi_t,J_t)$ only depends on the flux $u$ of the
symplectic isotopy $(\varphi_t)$. Hence we will denote its homology by
$FH(\bL,u)$ in the following.
\end{notn}

\subsection{Floer-Novikov complex over $\Lambda_{f^\star u}$}
\label{sec:surfstaru}

Thanks to Lemmata~\ref{lem:L0fini}, \ref{lem:dcarre}
and~\ref{lem:hom}, one can apply~\cite[Proposition 3.25]{Mihai} and
define a $\Lambda_{f^\star u}$-complex $C_\star (L,\varphi_t,J_t)$
spanned by the intersection points of $L$ and $\varphi_1(L)$ such that 
$f_\star: \pi_1(L) \rightarrow \pi_1(M)$ induces a morphism from
$C_\star (\bL,\varphi_t,J_t)$ to $C_\star (\bL,\varphi_t,J_t)$ via the
ring morphism $f: \Lambda_{f^\star u} \rightarrow \Lambda_u$.

The differential of this complex is defined for $x$ in $L_0 \cap L_1$
by 
$$\partial x = \sum_{y \in  L_0 \cap L_1} [x,y]^{\, \sim} \; y,$$
with $$[x,y]^{\,\sim} = \sum_{h \in \pi_1(L)} \#_2 \tcL_h (x,y) \; h,$$
where for any $h \in \pi_1(L)$, $\#_2 \tcL_h (x,y)$ is the set of
paths in $L$ from $x$ to $y$ which lift to the universal covering
$\tL$ of $L$ to paths from $h \tx$ to $\ty$.
One can prove as in Section~\ref{sec:invHam} that the homology of this
complex does not depend on the generic choice of the pair $(J_t,H_t)$.
We will denote this homology $FH(L,u)$.

\section{Floer homology and Novikov theory}
\label{sec:FloerNov}

In this section, we prove Theorem~\ref{thm:suitespecintro} of the
introduction. We deduce this theorem from an analogous result which relates 
$FH(\bL,u)$ and the Novikov homology $H(\bL, f^\star u)$ of $L$
associated to $f^\star u$ and the covering $\bL \rightarrow L$
(defined in Lemma~\ref{lem:rev} as the pull-back of the covering 
$\ctg \tM \rightarrow \ctg M$, see also Remark~\ref{rem:revbar}).

\subsection{Relation between $FH(\bL,u)$ and $H(\bL,f^\star u)$}
\label{sec:FHbetHb}

We first prove that the Floer-Novikov homology $FH(\bL,u)$ is invariant
by small rescaling of $u$:  

\begin{prop}  
\label{prop:42a} 
Let $u$ be an element of $H^1(M)$.\\
Assume that the Lagrangian submanifold $L$ is monotone on the loops
with Maslov number $N_L \geq 2$. 

Then there exists $\varepsilon > 0$ such that for every real number
$\sigma$ satisfying $| \sigma | < \varepsilon$, 
$$FH(\bL, (1+\sigma) u) \simeq FH(\bL,u).$$
\end{prop}

\begin{proof}
In the exact case, Damian uses the symplectic isotopy of $\ctg L$
spanned by a $1$-form in the class of $f^\star u$. Thanks to the
exactness of $L$, he extends it to $\ctg M$, so that the zeros of the
associated one-form are constant paths. 
This is particularly useful for the choice of a one-form representing
$u$ in the definition of $FH(\bL,(1+\sigma)u)$.

In the monotone case, we cannot use this result on the extension of
symplectic isotopies of $\ctg L$ to $\ctg M$. But, what we actually 
need for the proof (see \cite{Mihai} and also \cite{MR1308493}) 
is a symplectic isotopy with the property that the projection of 
the flow trajectories on the base space $M$ (and in particular the 
projection of the zeros of the one-form $\hnu$ associated to this 
isotopy) lie in ``small'' balls. 

More precisely, let $\alpha$ be a closed one-form in the class of $u$ 
and $J$ be a fixed compatible almost complex structure. Assume that
this almost complex structure induces a complete metric $g_J$ on 
$\ctg M$.

\begin{lem}
  There exists a constant $c_1 > 0$ such that, for every $\delta > 0$, 
  there exists a Hamiltonian 
  $$H: [0,1] \times \ctg M \rightarrow \rr$$
  with compact support and a family of almost complex structures
  $(J_t)$ satisfying $\|H_t\|_\epsilon < \delta$ and
  $\|J_t-J\|<\delta$ such that:
  \begin{description}
  \item[(i)]  the pair $(\alpha+dH_t,J_t)$ satisfies the
    transversality assumption;
  \item[(ii)]  the canonical projection $\pi$ of $\ctg M$ maps a zero
    $x$ of the one-form $\hnu$ associated to the symplectic isotopy
    spanned by $\alpha + dH_{t}$ into a ball in~$M$ centered in
    $\pi(x(0))$ with radius $c_1 \delta$. 
  \end{description}
\end{lem}

The norm $\| \; \|_\epsilon$ is the usual norm on the Hamiltonians 
used in transversality results (see \cite{MR948771},
\cite{MR1308493}): 
$$\|h\|_\epsilon = \sum_{k=0}^{\infty} \epsilon_k 
\|h\|_{C^k([0,1] \times \ctg M)}$$
where $\epsilon_k > 0$ is a sufficiently rapidly decreasing sequence.

\begin{proof}
The symplectic isotopy $\varphi^\alpha_t$ of $\ctg M$ spanned by the  
symplectic dual $X^\alpha$ of $\alpha$ can be written: 
$$\varphi^\alpha_t(p,q)=(p+t \alpha_{q}, q).$$

Note that the trajectory $\varphi^\alpha_t(p_0,q_0)$ of $X^\alpha$
in $\ctg M$ with initial condition the point $(p_0,q_0)$ lies in the
fiber of $q_0 \in M$.

Note also that $\varphi^\alpha_t$ does not necessarily satisfy the
transversality assumption between $L$ and $\varphi^\alpha_1(L)$ so 
that it may not be possible to use this isotopy for the description of 
$HF(L,u)$.

Nevertheless, by the transversality theorem (Theorem~\ref{thm:transv}) 
and Remark~\ref{rem:transversalite}, for every $\delta > 0$, there
exists a Hamiltonian $H: [0,1] \times \ctg M \rightarrow \rr$ with
compact support and a family of almost complex structures $(J_t)$
satisfying $\|H_t\|_\epsilon < \delta$ and $\|J_t-J\|<\delta$ and such
that the pair $(\alpha+dH_t,J_t)$ satisfies the transversality
assumption.

Let $\zeta(t)$ be the trajectory of $X^{\alpha+dH_t}$ with initial  
condition the point $(p_0,q_0)$ of $\ctg M$. Denote by $q(t)$ its
image in $M$ by the canonical projection $\pi: \ctg M \rightarrow M$
of the cotangent bundle. We have:
$$\frac{d}{dt}q(t) = T_{\zeta(t)} \pi \left(\frac{d}{dt}\zeta(t)
\right) 
= T_{\zeta(t)} \pi \left( X^{\alpha}(\zeta(t))+ X^{dH_t}(\zeta(t))
\right).$$
Since the isotopy $\varphi^\alpha_t$ spanned by $X^\alpha$ satisfies: 
$\pi(\varphi^\alpha_t)(p,q)= q$, 
$$T_{\zeta(t)} \pi \left( X^{\alpha}(\zeta(t))\right) = 0,$$ 
so that:
$$\frac{d}{dt}q_t = T_{\zeta(t)} \pi \left( X^{dH_t}(\zeta(t)) \right).$$
Moreover, for all $z \in \ctg M$, we have:
\begin{eqnarray*}
\|X^{dH_t}(z)\|^2 &=& d_z H_t \left( -J(z) X^{dH_t}(z) \right) \\
&\leq& \|d_z H_t \| \| -J(z) X^{dH_t}(z)\| = \|d_z H_t \| \| X^{dH_t}(z)\|
\end{eqnarray*}
for the norm associated to the scalar product $g_J$.
Hence, 
$$\|X^{dH_t}(z)\| \leq \|d_z H_t \| \leq \frac{\delta}{\epsilon_1}$$
(where $\epsilon_1$ is the first term of the sequence defining the norm 
 $\| \; \|_\epsilon$).

Notice that $\| T_{\zeta(t)} \pi\|$ is bounded on $\ctg M$: it is
bounded on each trivialising open set for the cotangent bundle 
$\ctg M \rightarrow M$; the base space $M$ being compact, it is
bounded on the whole $\ctg M$. 
Thus, there exists a constant $c_1$ (which does not depend on $H$)
such that, 
$$\left \| \frac{d}{dt}q_t \right \| \leq 
\| T_{\zeta(t)} \pi \| \left\| X^{dH_t}(\zeta(t)) \right \|
\leq c_1 \delta,$$
and consequently, for all  $\tau \in [0,1]$, we have:
$$d(q(\tau),q_0) \leq \int^\tau_0 \left \| \frac{d}{dt}q_t \right \|  
\leq \tau c_1 \delta.$$
This means that, on $[0,1]$, the trajectories of $X^{\alpha+dH_t}$ lie
in the fibers of the points of $M$ which belongs to the ball of radius 
$c_1 \delta$ centered in the projection of the initial condition 
(see figure \ref{fig:fig42a}).

\begin{figure}[htbp]
  \begin{center}
   \psfrag{q0}{$q_0$}
   \psfrag{M}{$M$}
   \psfrag{pt}{$\zeta(t)$}
   \psfrag{pti}{$\zeta(0)$}
   \psfrag{C4d}{$c_1 \delta$}
   \includegraphics[height=7cm]{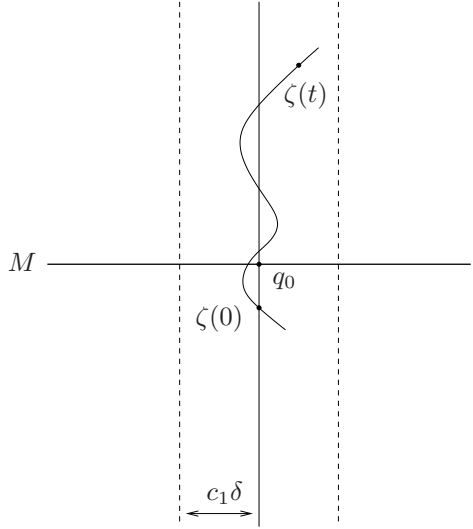}
   \caption{the trajectory $\zeta$} 
   \label{fig:fig42a}
  \end{center}
\end{figure}
This is in particular true for the zeros $x_i$ of the one-form
$\hnu$. 
\end{proof}

Let $\cV_i$ be a neighbourhood of each trajectory $x_i$ such that 
the projection of $\cV_i$ is contained in 
$B(\pi(x_i(0)),c_1 \delta)$. 
Denote $\cV$ the union of the $\cV_i$.\\

\begin{lem}
\label{lem:4.4}
There exists $c_2 > 0$ such that for all $z \in \Omega(L,L)$ whose
image is not contained in $\cV$,
$$\|z'(t) - X^{\alpha+dH_t}(z(t))\|_{L^2} \geq c_2.$$
\end{lem}

\begin{proof}
  The proof is analogous to \cite[Lemma 4.4]{Mihai}. Assume the
  contrary: there exists a sequence $(z_n) \in \Omega(L,L)$ of paths
  whose images are not contained in $\cV$ such that 
  $$\lim_{n \rightarrow +\infty} 
  \|z_n'(t) - X^{\alpha+dH_t}(z_n(t))\|_{L^2} = 0$$
  One has then to prove that this sequence admits a subsequence 
  converging to a zero of $\hnu$. 
  This contradicts the fact that the images of the $z_n$'s are not  
  contained in $\cV$. \end{proof}

We can now choose $\delta > 0$ small enough so that there exists a
closed one-form $\eta \in u$ such that 
$$\eta = 0 \mbox{ on } \bigcup_i B(\pi(x_i(0)),c_1 \delta).$$ 
We also fix $\varepsilon > 0$ such that 
$\varepsilon \|\eta\| < c_2/3$. In particular, the one-form (still
denoted $\eta$) lifted to $\ctg M$ is zero on $\cV$, a property that
we need in the computations (see Proposition~\ref{prop:comp4}, and
also~\cite{Mihai}, \cite{MR1308493}).

Choose a real number $\sigma < \varepsilon$ and consider the isotopy
$\varphi^{\alpha+ \sigma \eta + dH_t}_t$ spanned by 
$X^{\alpha+ \sigma \eta + dH_t}$. The constant $\varepsilon$ is chosen 
small enough so that $\varphi^{\alpha+ \sigma \eta + dH_t}_1(L)$ is
transverse to $L$.

Applying the transversality theorem~\ref{thm:transv} to 
$\alpha+ \sigma \eta +dH_t$, we deduce the existence of a compatible
almost complex structure $J'_t$ such that $\|J'_t-J\|< \delta$ and
such that the pair $(\alpha + \sigma \eta +dH_t, J'_t)$ satisfies the 
transversality assumption.

Since $\Lambda_u=\Lambda_{\tau u}$ for all $\tau > 0$, we can define
the $\Lambda_u$-complexes 
$$C_\star (\bL,\varphi_t^{\alpha+dH_t},J_t) \mbox{ and }
C_\star (\bL,\varphi_t^{\alpha+ \sigma \eta + dH_t},J'_t).$$
Let us prove that the homologies of this complexes are isomorphic. 
This will achieve the proof of Proposition~\ref{prop:42a}.\\

As in Section \ref{sec:invHam}, we define a morphism of complexes
associated to a homotopy between the pairs $(\alpha+dH_t,J_t)$ and
$(\alpha+dH_t+ \sigma \eta, J'_t)$. Let $\chi$ be a monotone
increasing function on $\rr$ that vanishes for $s \leq -R$ and is
equal to $1$ for $s \geq R$. Let $J_{s,t}$ be a homotopy of compatible 
almost complex structures such that $J_{s,t}=J_t$ for $s \leq -R$,
$J_{s,t}=J'_t$ for $s \geq R$ and $\|J_{s,t}-J\|< \delta$. The
homotopy is defined as: 
$$(\alpha+ \chi(s) \sigma \eta + dH_t, J_{t,s}).$$
Consider the space of solutions 
$$v: \rr \times [0,1] \longrightarrow \ctg M$$
of finite energy $E(v)$ (for the norm defined defined by $g_J$) of the
partial differential equation: 
\begin{equation} \label{eq:thm4.1}
  \frac{\partial v}{\partial s}
 + J_{s,t} \left ( \frac{\partial v}{\partial t} 
- X_{s,t}^{\alpha + \chi(s) \sigma \eta + dH_{s,t}}(v) \right ) = 0
\end{equation}
such that $v(s,i) \in L \text{ for } i =0,1$.

These solutions converge to an orbit of $X^{\alpha+dH_t}$
(respectively of $X^{\alpha+ \sigma \eta + dH_t}$) when $s$ goes to 
$- \infty$ (respectively $+ \infty$).

We define as in the previous sections the spaces
$\cM_{\chi,J_{(s,t)}}(x,y)$ of solutions between two orbits $x$ and 
$y$. By transversality, these spaces are manifolds the local dimension 
of which is given by the Maslov class. 

Let $\cM_{g,s}(x,y) \subset \cM_{\chi,J_{s,t}}(x,y)$ be the space of 
solutions that can be lifted to $\ctg \tM$ in paths from $\tx^g$ to 
$\ty$ (for fixed lifts $\tx$ and $\ty$). 
In the following, we prove that the zero-dimensional component of
$\cM_{g,s}(x,y)$ is compact, so that we can define a morphism 
$$\Gamma_\star: C_\star (\bL,\varphi_t^{\alpha+dH_t},J_t) \longrightarrow 
C_\star (\bL,\varphi_t^{\alpha+ \sigma \eta + dH_t},J'_t)$$
by
$$\Gamma_\star(x) = \sum_{g \in \pi_1(M), y} m_g(x,y) g y$$
where $m_g(x,y)$ is the cardinality (modulo $2$) of the space 
$\cM_{g,s}^0(x,y)$.\\

We prove now the compactness of $\cM_{g,s}^0(x,y)$ and we check that  
$$\sum_{g \in \pi_1(M), y} m_g(x,y) g \in \Lambda_u.$$

\begin{prop}
\label{prop:comp4}
For any fixed index $I$, the energy of any element of
$\cM_{g,s}^I(x,y)$ is bounded from above by some positive constant
$A$. 
\end{prop}

\begin{proof}
Let $v_1$ and $v_2$ be two elements of $\cM_{g,s}(x,y)$ with the same
Maslov index~$I$. If $\gamma: \rr \times [0,1] \rightarrow \ctg M$ is
the concatenation of $v_2$ and $\overline{v_1}$, then $\gamma$ is,
after reparametrization in the $s$ variable, a loop $\Omega(L,L)$
based at $x$. 

We have: 
\begin{eqnarray*}
\int \hnu \left( \frac{\partial \gamma}{\partial s} 
\right) \;ds
&=& - \int_{- \infty}^{+\infty} \int_0^1 \left \langle
\frac{\partial v_1}{\partial s}, J \left( \frac{\partial v_1}{\partial t} -
X^{\alpha + dH_t}(v_1) \right) \right \rangle  \;dt \;ds \\
& & + \int_{- \infty}^{+\infty} \int_0^1 \left \langle 
\frac{\partial v_2}{\partial s}, J \left( \frac{\partial v_2}{\partial t} - 
X^{\alpha + dH_t}(v_2) \right) \right \rangle  \;dt \;ds.
\end{eqnarray*}

We prove that:
\begin{description}
\item[1)] for any $s \in \rr$, 
  \begin{equation}
    \label{eq:1)}
    - \int_0^1 \left \langle
\frac{\partial v_1}{\partial s},
J \left( \frac{\partial v_1}{\partial t} - 
X^{\alpha + dH_t}(v_1) \right) \right \rangle \; dt \geq 
\frac{1}{3} \left \| \frac{\partial v}{\partial s} \right \|_{L^2}^2 \,;
  \end{equation}
\item[2)] for any $s \in \rr$,
  \begin{equation}
    \label{eq:2)}
    - \int_0^1 \left \langle
\frac{\partial v_2}{\partial s}, J \left( \frac{\partial v_2}{\partial t} - 
X^{\alpha + dH_t}(v_2) \right) \right \rangle \; dt
\leq \frac{5}{3} \left \| \frac{\partial v_2}{\partial s} 
\right\|_{L^2}^2.
  \end{equation}
\end{description}

Using 1) and 2), we will deduce that
\begin{equation}
  \label{eq:3)}
  \frac{1}{3} E(v_1) \leq  u(\gamma_0) + \frac{5}{3} E(v_2)
\end{equation}
and use this inequality to achieve the proof of 
Proposition~\ref{prop:comp4}.\\

\noindent
Proof of 1): If $v$ is an element of $\cM_{g,s}(x,y)$, we have:
\begin{align}
\int_0^1 \left \langle
\frac{\partial v}{\partial s}, \right.& \left. J \left( \frac{\partial v}{\partial t} - 
X^{\alpha + dH_t}(v) \right) \right \rangle \; dt = \nonumber \\
&=  \int_0^1 \left \langle
\frac{\partial v}{\partial s}, J \left( J_{s,t} \frac{\partial v}{\partial s} + X^{\alpha + \chi(s) \sigma \eta + dH_t}(v) -
X^{\alpha + dH_t}(v) \right) \right \rangle \; dt \nonumber  \\ 
&=  \int_0^1 \left \langle \frac{\partial v}{\partial s}, 
J  J_{s,t} \frac{\partial v}{\partial s} \right \rangle \; dt 
+ \int_0^1 \left \langle \frac{\partial v}{\partial s},  
J X^{\chi(s) \sigma \eta}(v) \right \rangle \; dt \nonumber \\
&= - \int_0^1 \left \langle J \frac{\partial v}{\partial s}, 
J_{s,t} \frac{\partial v}{\partial s} \right \rangle \; dt 
- \int_0^1 \omega_M \left ( \frac{\partial v}{\partial s},  
X^{\chi(s) \sigma \eta}(v)  \right) \; dt \nonumber \\
&= - \int_0^1 \left \| \frac{\partial v}{\partial s}\right \|^2 \; dt 
- \int_0^1 \left \langle J \frac{\partial v}{\partial s}, 
(J_{s,t}-J) \frac{\partial v}{\partial s} \right \rangle \; dt \nonumber \\
& \; \; \; \; - \int_0^1 \chi(s) \sigma \eta \left( \frac{\partial v}{\partial s}\right)
 \; dt. \label{eq:v}
\end{align}

\begin{itemize}
\item Assume firstly that $v(s,\cdot)$ takes values in $\cV$. Since
  $\eta$ vanishes on this neighbourhood of the zeros, we have:
\begin{align*}
\int_0^1 \left \langle \frac{\partial v}{\partial s}, \right. &
\left.  J \left( \frac{\partial v}{\partial t} - X^{\alpha + dH_t}(v) \right) 
\right \rangle \; dt = \\
&= - \int_0^1 \left \| \frac{\partial v}{\partial s} \right \|^2 \; dt 
+ \int_0^1 \left \langle J \frac{\partial v}{\partial s}, 
(J - J_{s,t}) \frac{\partial v}{\partial s} \right \rangle \; dt  \\
&\leq - \int_0^1 \left \| \frac{\partial v}{\partial s} \right \|^2 \; dt 
+ \delta \int_0^1 \left \| \frac{\partial v}{\partial s} \right \|^2 \; dt\\
&\leq - (1- \delta) \int_0^1 \left \| \frac{\partial v}{\partial s} \right \|^2 \; dt. 
\end{align*}
We can assume that $\delta < 2/3$, so that:
$$\int_0^1 \left \langle \frac{\partial v}{\partial s},
 J \left( \frac{\partial v}{\partial t} - X^{\alpha + dH_t}(v) \right) 
\right \rangle \; dt 
\leq - \frac{1}{3} \int_0^1 \left \| \frac{\partial v}{\partial s} \right \|^2 \; dt.$$

\item If $v(s,\cdot)$ does not take its values in $\cV$, by the same 
argument as in~\cite{Mihai} and using Lemma~\ref{lem:4.4}, we also 
have in this case that:  
$$\int_0^1 \left \langle \frac{\partial v}{\partial s},
 J \left( \frac{\partial v}{\partial t} - X^{\alpha + dH_t}(v) \right) 
\right \rangle \; dt 
\leq - \frac{1}{3} \left \| \frac{\partial v}{\partial s} \right \|_{L^2}^2,$$
if we choose $\delta$ small enough.
\end{itemize}

\noindent
Applying this to $v_1$ we get \ref{eq:1)} and integrating in the $s$
variable: 
$$ \frac{1}{3} E(v_1)
\leq - \int_{- \infty}^{+\infty} \int_0^1 \left \langle
\frac{\partial v_1}{\partial s}, J \left( \frac{\partial v_1}{\partial t} - 
X^{\alpha + dH_t}(v_1) \right) \right \rangle  \;dt \;ds.$$
\ \\

\noindent
Proof of 2): We use the relation (\ref{eq:v}):
\begin{align*}
- \int_0^1 \left \langle
\frac{\partial v_2}{\partial s}\right.&  \left. , J \left( \frac{\partial v_2}
{\partial t} - X^{\alpha + dH_t}(v_2) \right) \right \rangle \; dt =\\
&= \int_0^1 \left \| \frac{\partial v_2}{\partial s} \right \|^2 \; dt 
+ \int_0^1 \left \langle J \frac{\partial v_2}{\partial s}, 
(J - J_{s,t}) \frac{\partial v_2}{\partial s} \right \rangle \; dt \\
& \; \; \; \; + \int_0^1 \chi(s) \sigma (-\eta) \left(\frac{\partial v_2}{\partial s}\right) \; dt.
\end{align*}

\noindent
We distinguish again the two cases:
\begin{itemize}
\item either $v_2(s, \cdot)$ takes its values in $\cV$, so that we
  have: 
$$- \int_0^1 \left \langle
\frac{\partial v_2}{\partial s}, J \left( \frac{\partial v_2}{\partial t} - 
X^{\alpha + dH_t}(v_2) \right) \right \rangle \; dt
\leq (1+ \delta) \left \| \frac{\partial v_2}{\partial s} \right \|_{L^2}^2,$$

\item or $v_2(s, \cdot)$ does not take all its values in $\cV$ and we
  use Lemma~\ref{lem:4.4}. Firstly, we have:
$$ - \int_0^1 \left \langle
\frac{\partial v_2}{\partial s}, J \left( \frac{\partial v_2}{\partial t} - 
X^{\alpha + dH_t}(v_2) \right) \right \rangle \; dt
\leq (1+ \delta) \left \| \frac{\partial v_2}{\partial s} \right\|_{L^2}^2 + 
\frac{c_2}{3} \left \| \frac{\partial v_2}{\partial s} \right\|_{L^2}.$$
By a consequence of Lemma \ref{lem:4.4} (see \cite{Mihai}):
$$\frac{c_2}{3} \left \| \frac{\partial v_2}{\partial s} \right\|_{L^2}
 \leq \left( \frac{2}{3} - \delta \right) 
\left \| \frac{\partial v_2}{\partial s} \right\|_{L^2}^2.$$
\end{itemize}

\noindent
In both cases, we have (\ref{eq:2)}) and by integration: 
$$-  \int_{- \infty}^{+\infty} \int_0^1 \left \langle
\frac{\partial v_2}{\partial s}, J \left( \frac{\partial v_2}{\partial t} - 
X^{\alpha + dH_t}(v_2) \right) \right \rangle  \;dt \;ds
\leq \frac{5}{3} E(v_2).$$

To prove (\ref{eq:3)}), we use that  
$$ \int_{- \infty}^{+\infty} \hnu \left( \frac{\partial \gamma}
{\partial \sigma} \right)$$
is the value of the one-form $\hnu$ on the loop $\gamma$ based 
in $x$. Thanks to (\ref{eq:alt2}), this term is equal to $u(\gamma_0)$ 
if the two solutions have the same Maslov class.\\

By assumption, $v_1$ and $v_2$ can be lifted to paths between $\tx^g$ 
and $\ty$, the loop $\gamma_0$ is homotopic to the constant loop 
based in $x$. 

We thus have, fixing an element $v_0$ in $\cM^I_{g,s}(x,y)$, for all 
$v$ of $\cM^I_{g,s}(x,y)$, 
$$E(v) \leq 5 E(v_0).$$
To end the proof of Proposition~\ref{prop:comp4}, 
we choose $A= 5  E(v_0)$. \end{proof}

Looking at the proof of Proposition~\ref{prop:comp4}, 
we see that the sum  
$$\sum_{g \in \pi_1(M)} m_g(x,y) g \in \Lambda_u.$$
Indeed, if $v_1$ can be now lifted as a path from $\tx^{g_1}$ to $\ty$
and $v_2$ as a path from $\tx^{g_2}$ to $\ty$, we have 
$u(\gamma_0)= u(g_1)-u(g_2)$, so that (\ref{eq:3)}) becomes: 
$$\frac{1}{3} E(v_1) \leq u(g_1)-u(g_2) + \frac{5}{3} E(v_2).$$
This implies that if $v_0$ is a fixed element of $\cM_{g_0,s}(x,y)$, 
we have then for all $v$ of $\cM_{g,s}(x,y)$ with $u(g) < C$,
$$E(v) \leq 3(C - u(g_0))+ 5 E(v_0)$$
and this prove that 
$$ \bigcup_{u(g)<C} \cM^0_{g,s}(x,y) \subset \cM_{\chi,J_{s,t}}(x,y;A)$$
for the positive constant $A = 3(C - u(g_0))+ 5 E(v_0))$.\\

The map $\Gamma_\star$ is a morphism of complexes. This is a
consequence of the compactification of the one-dimensional component
of $\cM_{g,s}(x,y)$ by broken trajectories $(v^1,v^2)$, where
\begin{itemize}
\item one of the $v^i$'s satisfies (\ref{eq:thm4.1});
\item the other is solution of the Floer equation corresponding to
  $(\alpha+dH_t,J_t)$ or $(\alpha+dH_t+ \sigma \eta, J'_t)$.
\end{itemize}

As in Theorem~\ref{thm:invar-hamilt}, we use the usual methods of 
Floer theory to prove that the morphism of complexes $\Gamma_\star$ 
that induces an isomorphism in homology. \end{proof}

Now, we prove that, with an additional assumption on the Maslov number of
$L$, for $\sigma$ small enough, $FH(\bL,\sigma u)$ is the Novikov
homology of $L$ associated to $f^\star u$ and the covering 
$\bL \rightarrow L$.

\begin{prop}
\label{prop:42b}
Let $u$ be an element of $H^1(M)$.\\
Assume that the Lagrangian submanifold $L$ is monotone on the loops 
and that its Maslov number $N_L$ satisfies 
$$N_L \geq \dim(M)+2.$$
Then there exists $\varepsilon > 0$ (depending on $u$) such that 
for all real number $\sigma$ satisfying $| \sigma | < \varepsilon$,
$$FH(\bL, \sigma u) \simeq H(\bL,f^\star u).$$
\end{prop}

\begin{proof}
We begin (as in \cite{Mihai}) by substituting in the proof of 
Proposition~\ref{prop:42a} the one-form  $\alpha \in u$ by an 
exact one-form $dg$ where $g:M \rightarrow \rr$. 
For $\sigma$ small enough, we obtain a morphism of
$\Lambda_u$-complexes: 
$$\Gamma^0_\star: C_\star(\bL,\varphi_t^{dg+dH_t},J_t) \longrightarrow 
C_\star(\bL,\varphi_t^{dg+\sigma \eta + dH_t},J'_t) $$
which induces an isomorphism in homology. The first complex is a 
$\Lambda$-complex whose coefficients have been extended to
$\Lambda_u$. There is a natural isomorphism 
$$C_\star(\bL,\varphi_t^{dg+dH_t},J_t) \simeq  
\Lambda_u \otimes_\Lambda  C_\star(\bL,\varphi_t^{dg+dH_t},J_t).$$

By Hamiltonian invariance, we know that the homotopy type of the
complex $C_\star(\bL,\varphi_t^{dg+dH_t},J_t)$ does not depend on a
regular choice of the pair $(H^0_t,J^0_t)$. 

In order to define an isomorphism between $FH(\bL, \sigma u)$ and 
$H_\star (\bL,u)$, we use an other choice of Hamiltonian. For that
purpose, in the monotone case, we need to adapt a construction of Oh
(\cite{MR1389956}).
We consider a local Floer homology, namely the Floer homology in a
Darboux neighbourhood $\cU$ of $L$ in~$\ctg M$. Considering $\cU$ as
the neighbourhood of the zero section in $\ctg L$, we can define the
Hamiltonian $\cH = h \circ \pi_L$ on $\cU$, where $h$ is a Morse
function on $L$ and $\pi_L: \ctg L \rightarrow L$ is the canonical
projection of the cotangent.

If $h$ is small enough (in the $C^2$-topology) and if its gradient for
a metric on $L$ is Morse-Smale, the local Floer complex is spanned by
the intersection points of $L$ and its displacement $L+dh$ and we have
a bijection between the $J$-holomorphic strips of the Floer homology
and the trajectories of the gradient of $h$ which define the Morse
differential. 

To go back to $\ctg M$, it is then enough to extend the Hamiltonian
$\cH$ to $\ctg M$ (setting $\cH = 0$ outside a neighbourhood containing
$\cU$).  
Oh has proved that under the assumption that $L$ is monotone (on the
disks) in $\ctg M$ and $N_L \geq \dim(M) + 2$, a Floer trajectory in
$\ctg M$ stays in the Darboux neighbourhood and hence the trajectories
that define the ``global'' Floer differential are those that were
already counted in the local differential. 

The end of the proof is similar to \cite{Mihai}. 
The Novikov ring which defines the Novikov homology associated to
$f^\star u$ and the covering $\bL \rightarrow L$ is $\Lambda_u$, so
that the Morse complex above is exactly 
$C_\star(\bL \rightarrow L, h, \xi)$. 
The $\Lambda_u$-complexes
$$\Lambda_u \otimes_\Lambda  C_\star(L,\varphi_t^{dg+dH_t},J_t) 
\mbox{ and } \Lambda_u \otimes_\Lambda C_\star(\bL \rightarrow L,h,\xi)$$
are homotopy equivalent, so that the homologies 
$$FH(\bL, \sigma u) \simeq H(\bL,u)$$
are isomorphic.  \end{proof}

Propositions \ref{prop:42a} and \ref{prop:42b} imply that the set 
$$\{ \sigma \in ]0;+\infty[ \;|\; FH(\bL,u) \simeq H_{\star}(\bL,f^\star u) \}$$
is nonempty, open and closed, hence equal to $]0;+\infty[$ so that we
have proved the following theorem: 

\begin{thm}
\label{thm:4.1}
Let $u$ be an element of $H^1(M)$.\\
Assume that the Lagrangian submanifold $L$ is monotone on the loops 
and its Maslov number $N_L$ satisfies 
$$N_L \geq \dim(M)+2.$$
Then the Floer homology $FH(\bL,u)$ is isomorphic to the Novikov homology 
$H(\bL,f^\star u)$.
\end{thm}

If we only assume that $N_L \geq 2$, we do not necessarily have an
isomorphism between the Floer homology and the Novikov homology of 
$L$. This is also the case in usual Floer theory, but we have the  
spectral sequence described by Biran in \cite{Birannew} to relate it
to the singular homology of the Lagrangian submanifold. We can also
define in the monotone case a spectral sequence whose first page is
the Novikov homology of $L$ and that converges to the Floer-Novikov
homology. The following theorem  gives a precise description of this
spectral sequence. 

Let $\bA$ be the $\Lambda_u$-module $\bA = \Lambda_u[T,T^{-1}]$ of Laurent
polynomials with coefficients in $\Lambda_u$. Let 
$\tau_i :\bA \rightarrow \bA$ be the multiplication by~$T^i$. 
We define the degree of $T$ to be $N_L$. Then 
$$\bA = \bigoplus_{i \in \zz} \bA^i,$$
where $\bA^i = \Lambda_u \, T^{i/N_L}$ if $i \equiv 0 \mod N_L$ and  
$\bA^i = \{0\}$ otherwise.

\begin{thm}
\label{thm:suitespec} 
  There exists a spectral sequence $\{E_{r}^{p,q},d_{r}\}$ satisfying 
the following properties:
\begin{description}
\item[(1)] $E_{0}^{p,q} = C_{p+q-pN_L}(\bL,f^\star u) \otimes \bA^{pN_{L}}$ and 
$d_0=\partial_0 \otimes 1$;
\item[(2)] $E_{1}^{p,q} = H_{p+q-pN_L}(\bL,f^\star u) \otimes \bA^{pN_{L}}$ and  
$d_1=[\partial_1] \otimes \tau$ where 
$$[\partial_1]: H_{p+q-pN_L}(\bL,f^\star u) \longrightarrow H_{p+1+q-(p+1)N_L}(\bL,f^\star u)$$
is induced by $\partial_1$;
\item[(3)] For all $r \geq 1$, $E_{r}^{p,q}$ can be written  
  $E_{r}^{p,q}=V_{r}^{p,q} \otimes \bA^{pN_{L}}$ with 
  $d_{r}=\delta_{r} \otimes \tau_{r}$, 
  $V_{r}^{p,q}$ are modules on $\Lambda_u$, 
  $\delta_{r}: V_{r}^{p,q} \rightarrow V_{r}^{p+r,q-r+1}$ are
  morphisms and satisfy $\delta_{r} \circ \delta_{r} = 0$. Moreover,   
  $$V_{r+1}^{p,q}= \frac{\ker (\delta_{r}: V_{r}^{p,q} \rightarrow
    V_{r}^{p+r,q-r+1})}{\im (\delta_{r}: V_{r}^{p-r,q+r-1} \rightarrow
    V_{r}^{p,q})};$$
\item[(4)] $\{E_{r}^{p,q},d_{r}\}$ collapses at page $\kappa + 1$, 
  where $\kappa = [\frac{\dim(L)+1}{N_L}]$ and the spectral sequence  
  converges to $FH(\bL,u)$, i.e.  
$$\bigoplus_{p+q=\ell} E_{\infty}^{p,q} \cong FH^{\ell(\modulo N_{L})}(\bL,u).$$
\item[(5)]  For all $p \in \zz$, 
$\displaystyle \bigoplus_{q \in \zz} E_{\infty}^{p,q} \cong FH(\bL,u).$
\end{description}
\end{thm}

\begin{proof}
In order to describe the spectral sequence, we look at the proof of 
Proposition~\ref{prop:42b} and the Hamiltonian $\cH$ defined with the 
$C^2$-small function~$h$. We suppose, as in \cite{Birannew}, that $h$ 
has exactly one relative minimum $x_0$ and we use $x_0$ as base  
point for the Floer complex so that we can fix the grading
by~$\zz/N_L$. 
As we have the decomposition (see \cite{MR1389956}):
$$C_{i(\modulo N_L)}(\bL,\varphi_t^{H^0_t},J^0_t) = 
\bigoplus_{j \equiv i(\modulo N_L)} C_j(\bL,f^\star u).$$
we can decompose the differential  
$$\partial :C_{\star \modulo N_L}(\bL,\varphi_t^{H^0_t},J^0_t) \longrightarrow 
C_{\star +1 \modulo N_L}(\bL,\varphi_t^{H^0_t},J^0_t)$$ 
in 
$\displaystyle \partial = \sum_{j \in \zz} \partial_j$ with 
$$\partial_j:  C_\star (\bL,f^\star u) \rightarrow C_{\star+1-j N_L} (\bL,f^\star u).$$
Moreover, by the index computations of \cite{MR1389956}, 
$$\partial_j=0 \mbox{ if } j<0 \mbox { or } j> \kappa = \left [
  {\frac{\dim L +1}{N_L} } \right],$$ 
so that  
$$\partial = \partial_0 + \cdots + \partial_\kappa.$$
The differential $\partial_0$ counts the trajectories that stay in 
the neighbourhood $\cU$, it corresponds to the differential of the 
local Floer homology and as before, the homology of the complex 
$(C_\star(\bL,f^\star u),\partial_0)$ can be identified to the Novikov
homology $H(\bL,f^\star u)$. 
The other operators $\partial_1, \ldots, \partial_\kappa$ count the 
trajectories that leave the neighbourhood $\cU$.

In order to define and prove the properties of the spectral sequence,
it suffices to substitute the coefficients in $\zz/2$ in the proof of
\cite[Theorem 5.2]{Birannew} by coefficients in $\Lambda_u$.
\end{proof}

\begin{rem}
  Note that if $N_L \geq \dim(M)+2$, the spectral sequence collapses at 
  page $\kappa+1=1$ and we recover Theorem~\ref{thm:4.1}.
\end{rem}

\subsection{Proof of Theorems~\ref{thm:suitespecintro} and~\ref{thm:princ}}
\label{sec:preuve1.4et1.5}

\noindent
\textit{\textbf{Proof of Theorem~\ref{thm:suitespecintro}.}}
As in \cite[Section 4.1]{Mihai}, it is a consequence of the 
proofs of Propositions~\ref{prop:42a}, \ref{prop:42b} and
Theorem~\ref{thm:suitespec}. Thanks to Proposition~\ref{prop:comp4},
we can define for $\sigma$ small enough a lift of the morphism of
$\Lambda_{f^\star u}$-complex $\Gamma_\star$ between $C_\star
(L,\varphi_t^{\alpha+dH_t},J_t)$ and  
$C_\star (L,\varphi_t^{\alpha+ \sigma \eta + dH_t},J'_t)$ which
induces an isomorphism in homology:
$$FH(L,u) \simeq FH(L, (1+\sigma) u).$$
In order to relate $FH(L, \sigma u)$ and $H_\star (L,f^\star u)$, we
use a lift of the morphism $\Gamma^0_\star$ defined in the proof of
Proposition~\ref{prop:42b} and a spectral sequence analogous to the
one in Theorem~\ref{thm:suitespec} with the first page expressed in
terms of $H_\star(L,f^\star u)$ and the 
$\Lambda_{f^\star u}$-module $A= \Lambda_{f^\star u} [T,T^{-1}]$.
\hfill \qedsymbol\\

When $M$ is the total space of a fibration on the circle, 
Theorem~\ref{thm:suitespecintro} enables us to prove
Theorem~\ref{thm:princ} in the case where 
$f_{\star}: \pi_{1}(L) \longrightarrow \pi_{1}(M)$ is
surjective. Thanks to Remark~\ref{rem:hypsurj}, this will be enough 
to prove the theorem under the hypothesis  $N_L \geq \dim(M)+1$ or 
($N_L = \dim(M)$ and $[\pi_1(M):\pi_1(L)]$ is finite).\\

\noindent
\textit{\textbf{Proof of Theorem~\ref{thm:princ} in the case
    \mathversion{bold} $N_L \geq \dim(M)+1$. \mathversion{normal}}}  
Since the manifold $M$ is the total space of a fibration on the
circle, there exists a closed one-form~$\alpha$ that does not vanish
on $M$. Consider the symplectic isotopy $\varphi_t$ of $\ctg M$
spanned by $X^\alpha$: 
$$\varphi_t(p,q) = (p+t \alpha_q,q).$$
For $T$ large enough, $\varphi_T(L) \cap L = \emptyset$ and without
restricting generality we can assume that $T=1$.  
Then the Floer complex defined in Section \ref{sec:cpxFloer} is empty 
and the Floer homology $FH(L,u)$ is trivial.

But, as $u \neq 0$ and as $f_{\star}: \pi_1(L) \rightarrow \pi_1(M)$
is surjective, $f^{\star}u \neq 0$ and by
Proposition~\ref{prop:prelibre}, if the presentation of the
fundamental group of $L$ satisfies $p-q \geq 2$ or if the fundamental
group is a free product of two non trivial groups, then
$H_1(L,f^\star u) \neq 0$. 

If $N_L \geq \dim(M)+2$, then by Theorem~\ref{thm:suitespecintro}, the 
Floer homology of $L$ is isomorphic to the Novikov homology 
$H_\star(L, f^\star u)$ and this contradicts 
$H_1(L,f^\star u) \neq 0$.
If $N_L=\dim(M)+1 \geq 2$, then the spectral sequence defined in 
Theorem~\ref{thm:suitespec} collapses at page $\kappa+1=2$ and 
converges to the Floer homology of $L$.
But for $p=0$ and $q=1$, 
$$[\partial_1]: H_1(L,f^\star u) \longrightarrow H_{2-N_L}(L,f^\star u)=\{0\}$$
and 
$$[\partial_1]: H_{N_L}(L,f^\star u)=\{0\} \longrightarrow H_1(L,f^\star u),$$
hence $E^2_{0,1} = H_1(L,f^\star u)$.
We get also a contradiction in this case. \hfill \qedsymbol\\

\noindent
\textit{\textbf{Proof of Theorem~\ref{thm:princ} in the case
\mathversion{bold}  $N_L \geq \dim(M)$.\mathversion{normal}}} 
  We can even extend the result to the case $N_L = \dim(M)$ 
  thanks to a property of the Novikov homology 
  $H_\star(L,f^\star u)$ (see Latour~\cite{MR1320607} or
  Levitt~\cite{MR884804}): since $f^{\star}u \neq 0$, there exists in
  the class of $f^{\star}u$ a one-form $\alpha'$ which has no critical
  point of index $0$ or $\dim(M)$. In particular, the groups
  $H_0(L,f^\star u)$ and $H_{\dim(M)}(L,f^\star u)$ are trivial. 
  
  By Theorem~\ref{thm:suitespecintro}, there exists a spectral sequence
  converging to the Floer homology of $L$ and whose first page can be
  described with the Novikov homology of $L$. 
  In the case $N_L = \dim(M) \geq 2$, the spectral sequence collapses
  at the second page,
  $$[\partial_1]: H_1(L,f^\star u) \longrightarrow 
  H_{2-N_L}(L,f^\star u)=\{0\}$$
  and 
  $$[\partial_1]: H_{N_L}(L,f^\star u)=\{0\} \longrightarrow 
  H_1(L,f^\star u),$$
  so that $E^2_{0,1} = H_1(L,f^\star u)$. 
  As before, this leads to a contradiction. \hfill \qedsymbol\\

\section*{Appendix: Proof of Lemma~\ref{lem:rev}}
\label{preuvelemmerev}

 \begin{description}

  \item[(i)] Assume first that $\bL$ is path-connected. Choose a base
    point $\ell$ in $L$ and let $m=i(\ell)$ be its image in $\ctg M$.
    Choose also a lift $\tell$ of $\ell$ in $\bL$. 
    Let $g$ be an element of $\pi_1(M)$. 
    As $\bL$ is path-connected, there exists a path $\tgamma$ from 
    $\tell$ to $\tell^g$ in $\bL$. 
    The image $(\pi \circ \tp \circ \ti)(\tgamma)$ of that path in $M$
    is a loop representing $g$ and its image in $L$ is thus a loop
    $\gamma$ such that $f_\star([\gamma])=g$.

    Conversely, assume that the map $\pi_1(L) \rightarrow \pi_1(M)$ is 
    surjective. 
    As $L$ is supposed to be path-connected, to prove that $\bL$ is  
    path-connected, it is enough to prove that two points in the same  
    fiber can be joined by a path in $\bL$. 
    Let $\tell_1$ and $\tell_2$ two elements of $\bL$ in the fiber of
    the point $\ell$ of $L$. 
    The two points $\ti(\tell_1)$ and $\ti(\tell_2)$ of $\ctg \tM$ can
    be joined in $\ctg \tM$ by a path which projection on $\ctg M$ is
    a loop $c$ such that 
    $\ti(\tell_2)= \left({\ti(\tell_1)}\right)^{[c]}$. 
    But by assumption, the element $[c]$ of the group $\pi_1(M)$ has
    an antecedent in $\pi_1(L)$. 
    Let $\gamma$ be a loop in $L$ based in $\ell$ representing this
    antecedent. 
    If $\tgamma$ is the lift with starting point $\tell_1$ of $\gamma$
    in $\bL$, its endpoint $\tell_1^{\;[\gamma]}$ must have 
    $\left({\ti(\tell_1)}\right)^{[c]}$ as image by $\ti$, and this
    proves that $\tgamma$ is a path from $\tell_1$ to $\tell_2$.

  \item[(ii)] By definition, $\bL$ fits into the following commutative 
    diagram:
    $$
    \xymatrix{
      \bL \ar[r]^-{\ti} \ar[d] & \ctg \tM \ar[d]^-{\tp} \\
      L   \ar[r]^-{i}    & \ctg M
    }
    $$
    As $\ctg \tM$ is simply connected, the diagram induced on the 
    fundamental groups gives that 
    $\im (\pi_1(\bL) \rightarrow \pi_1(L))$ 
    is included in $K$.
    Conversely, any loop of $L$ whose image in $M$ is homotopic to the 
    constant path can be lifted to $\ctg \tM$ in a loop of $\bL$.

 \item[(iii)] The embedding is monotone because the symplectic
   structure (and the Liouville form) on $\ctg \tM$ are obtained by
   taking the pull-back of those of~$\ctg M$.
  \end{description}

\bibliographystyle{plain}
\bibliography{Biblio}

\ \\
Institut de recherche mathématique avancée, Université Louis Pasteur\\
7, rue René Descartes, 67 084 Strasbourg, France.\\
e-mail address: gadbled@math.u-strasbg.fr

\end{document}